\documentclass[12pt,a4paper]{amsart}
\usepackage{amscd}
\usepackage[latin2]{inputenc}
\usepackage{t1enc}
\usepackage[mathscr]{eucal}
\usepackage{indentfirst}
\usepackage{graphics}
\usepackage{pict2e}
\usepackage{epic}
\usepackage{epstopdf} 
\numberwithin{equation}{section}

\usepackage[margin=1in]{geometry}   
\usepackage{graphicx}               
\usepackage{enumerate}          	
\usepackage{amsmath}            	
\usepackage{amssymb}				
\usepackage{bm}						
\usepackage{mathtools}              
\usepackage{amsfonts}               
\usepackage{amsthm}                 
\usepackage{pgfplots}           	
\usepackage{hyperref}	   			
\usepackage[titletoc]{appendix} 	
\usepackage{tikz}
\usepackage{tikz-cd}
\hypersetup{ 
	urlbordercolor=red,
	linkbordercolor=red,
	citebordercolor=red,
	runbordercolor=red,
	citebordercolor=red,
	filebordercolor=red
}

\theoremstyle{plain}
\newtheorem{thm}{Theorem}[section]
\newtheorem{lem}[thm]{Lemma}
\newtheorem{prop}[thm]{Proposition}
\newtheorem{cor}[thm]{Corollary}
\newtheorem{theorem}{Theorem}

\theoremstyle{definition}

\newtheorem{defn}[thm]{Definition}

\theoremstyle{remark}

\newcommand{\ZZ}{\mathbb{Z}}      
\newcommand{\NN}{\mathbb{N}}    
\newcommand{\QQ}{\mathbb{Q}} 
\newcommand{\Fields}{\mathbf{Fields}}      
\newcommand{\Set}{\mathbf{Sets}}      
\newcommand{\Rings}{\mathbf{Comm Rings}}      
\newcommand{\AbGrps}{\mathbf{Ab Grps}}      
\newcommand{\bmu}{\bm{\mu}}     
\DeclareMathOperator{\thh}{th}

\DeclareMathOperator{\Id}{Id}

\newcommand{\mfp}{\mathfrak{p}} 

\DeclareMathOperator{\Tw}{Tw}

\DeclareMathOperator{\Gal}{Gal}
\DeclareMathOperator{\Aut}{Aut}

\DeclareMathOperator{\Inv}{Inv}
\DeclareMathOperator{\Hom}{Hom}

\DeclareMathOperator{\Spec}{Spec}
\DeclareMathOperator{\Pic}{Pic}

\DeclareMathOperator{\sep}{sep}

\DeclareMathOperator{\tors}{tors}
\DeclareMathOperator{\Tors}{Tors}

\DeclareMathOperator{\alg}{alg}

\DeclareMathOperator{\Imm}{Im}
\DeclareMathOperator{\Cl}{Cl}
\DeclareMathOperator{\Div}{Div}
\DeclareMathOperator{\dv}{div}
\DeclareMathOperator{\Cart}{Cart}
\DeclareMathOperator{\ev}{ev}
\DeclareMathOperator{\nm}{nm}

\DeclareMathOperator{\Frac}{Frac}

\DeclareMathOperator{\chara}{char}

\begin{document}

\title{Degree One Milnor $K$-invariants of Groups of Multiplicative Type}

\author[A. Wertheim]{Alexander Wertheim}


\email{awertheim@math.ucla.edu}



\begin{abstract} Let $G$ be a commutative affine algebraic group over a field $F$, and let $H \colon \Fields_{F} \to \AbGrps$ be a functor. A (homomorphic) $H$-invariant of $G$ is a natural transformation $\Tors(-, G) \to H$, where $\Tors(-, G)$ is the functor $\Fields_{F} \to \AbGrps$ taking a field extension $L/F$ to the group of isomorphism classes of $G_{L}$-torsors over $\Spec(L)$. The goal of this paper is to compute the group $\Inv_{\hom}^{1}(G, H)$ of $H$-invariants of $G$ when $G$ is a group of multiplicative type, and $H$ is the functor taking a field extension $L/F$ to $L^{\times} \otimes_{\ZZ} \QQ/\ZZ$.    
\end{abstract}

\maketitle

\section{Introduction} 
Let $G$ be an affine algebraic group over a field $F$ (of arbitrary characteristic), and let $\Fields_{F}$ denote the category of field extensions of $F$. Let 
\[H \colon \Fields_{F} \longrightarrow \AbGrps\]
be a functor. In \cite{GMS}, an {\bf $H$-invariant of $G$} is defined to be a natural transformation of set-valued functors   
\[I \colon \Tors(-, G) \longrightarrow H\]
where $\Tors(-, G)$ is the functor from $\Fields_{F}$ to $\Set$ taking a field extension $L/F$ to $\Tors(L, G_{L})$, the set of isomorphism classes of $G_{L}$-torsors over $\Spec(L)$. Invariants were first introduced by Serre in \cite[Section 6]{SerreInvariants}, where he defined invariants in the case when $H$ is a (Galois) cohomological functor. 

There is another type of invariant that one may consider, however. Namely, since any affine group scheme over $F$ may be viewed as a functor from $F$-algebras to groups, we define a {\bf type-zero} $H$-invariant of $G$ to be a natural transformation of set-valued functors $G \to H$, where by $G$ we mean the restriction of $G$ to $\Fields_{F}$. We denote the group of type-zero $H$-invariants of $G$ by $\Inv^{0}(G, H)$. To distinguish the invariants introduced in the previous paragraph from type-zero invariants, we will call them {\bf type-one} invariants, and we denote the group of type-one $H$-invariants of $G$ by $\Inv^{1}(G, H)$.

In this paper, we study type-one invariants when $G$ is an algebraic group of multiplicative type, i.e. when $G$ is a twisted form of a diagonalizable group. We note that every torus is a group of multiplicative type; in general, groups of multiplicative type need not be smooth or connected. We consider a slightly more restrictive class of invariants than those introduced above, however. If $G$ is \emph{commutative}, then for any affine $F$-scheme $X$, the pointed set $\Tors(X, G)$ can be given the structure of an abelian group. Since groups of multiplicative type are commutative, we may view the functor $\Tors(-, G)$ as a functor from $\Fields_{F}$ to $\AbGrps$. Accordingly, we will focus our attention on invariants which are morphisms of \emph{group-valued} functors. We will call such invariants {\bf homomorphic}, and we denote the subgroup of homomorphic type-one $H$-invariants of $G$ by $\Inv^{1}_{\hom}(G, H)$. Likewise, one may consider homomorphic type-zero invariants, which we will similarly denote $\Inv^{0}_{\hom}(G, H)$.

The goal of this paper is to determine $\Inv^{1}_{\hom}(G, K^{M}_{1} \otimes_{\ZZ} \QQ/\ZZ)$, the group of (type-one) {\bf degree one Milnor $K$-invariants of $G$}, where $K^{M}_{i}$ denotes the functor sending a field extension $L/F$ to the $i^{\thh}$ Milnor $K$-group of $L$ (see \cite{Milnor}); we recall that $K^{M}_{0}(L) = \ZZ, K^{M}_{1}(L) = L^{\times}$. For any $n \in \NN$, let $K^{M}_{1}/n$ denote the functor $K^{M}_{1} \otimes_{\ZZ} \ZZ/n\ZZ$. The embedding $\ZZ/n\ZZ \hookrightarrow \QQ/\ZZ$ induces a morphism of functors $\iota_{n} \colon K^{M}_{1}/n \to K^{M}_{1} \otimes_{\ZZ} \QQ/\ZZ$; likewise, if $n$ and $m$ are positive integers such that $n$ divides $m$, then the embedding $\ZZ/n\ZZ \to \ZZ/m\ZZ$ sending $[1]_{n}$ to $[m/n]_{m}$ induces a morphism of functors $\beta_{n, m} \colon K^{M}_{1}/n \to K^{M}_{1}/m$. One may check that the collection of functors $\{K^{M}_{n} \otimes_{\ZZ} \QQ/\ZZ\}_{n \in \NN}$ defines the data of a cycle module in the sense of Rost (see \cite{Rost}), as does $\{K^{M}_{n} \otimes_{\ZZ} \ZZ/m\ZZ\}_{n \in \NN}$ for any $m \in \NN$. The functors $K^{M}_{1} \otimes_{\ZZ} \QQ/\ZZ$ and $K^{M}_{1}/m$ respectively form the first graded components of these cycle modules. 

As we will explain in Section \ref{BundlesAreTorsors}, a classic Kummer theory argument shows that there is an isomorphism $\Sigma_{n} \colon K^{M}_{1}/n \to \Tors(-, \bmu_{n, F})$ of group valued functors. On the other hand, any $\chi \in \Hom(G, \bmu_{n, F}) = G^{\ast}[n]$ gives rise to a morphism of group-valued functors $\Tors_{\ast}(\chi) \colon \Tors(-, G) \to \Tors(-, \bmu_{n, F})$. Thus, we may associate to any element $\chi \in G^{\ast}[n]$ a homomorphic invariant $I_{\chi} \in \Inv^{1}_{\hom}(G, K^{M}_{1}/n)$ which is the composition of $\Tors_{\ast}(\chi)$ with $\Sigma_{n}^{-1}$. This leads us to our first main theorem.

\begin{theorem}[\ref{TypeOneGMTModN}] \label{BigA} 
	The map $\Phi(G, n) \colon G^{\ast}[n] \to \Inv^{1}_{\hom}(G, K^{M}_{1}/n)$ sending $\chi$ to $I_{\chi}$ is a group isomorphism. 
\end{theorem}    

The composition of any such $I_{\chi}$ with $\iota_{n}$ produces an element of $\Inv^{1}_{\hom}(G, K^{M}_{1} \otimes_{\ZZ} \mathbb{Q}/\mathbb{Z})$, and so defines a group homomorphism $\tilde{\Phi}(G, n) \colon G^{\ast}[n] \to \Inv^{1}_{\hom}(G, K^{M}_{1} \otimes_{\ZZ} \mathbb{Q}/\mathbb{Z})$ for each $n \in \NN$. Passing to the colimit as $n$ varies, we obtain a universally induced group morphism $\Phi(G) \colon G^{\ast}_{\tors} \to \Inv^{1}_{\hom}(G, K^{M}_{1} \otimes_{\ZZ} \mathbb{Q}/\mathbb{Z})$.

\begin{theorem}[\ref{TypeOneGMTQZ}] \label{BigB} 
The map $\Phi(G) \colon G^{\ast}_{\tors} \to \Inv^{1}_{\hom}(G, K^{M}_{1} \otimes_{\ZZ} \mathbb{Q}/\mathbb{Z})$ is a group isomorphism. 
\end{theorem}

Our proofs of Theorems \ref{TypeOneGMTModN} and \ref{TypeOneGMTQZ} depend critically on the determination of homomorphic type-zero invariants for tori with values in $K^{M}_{1}/n$ for each $n \in \NN$. The following result was proven by Merkurjev (cf. \cite[Corollary 3.7]{MerkurjevIAG}) in the case when the characteristic of $F$ does not divide $n$; we give a proof in this paper which holds independent of the characteristic of $F$.

\begin{theorem}[\ref{TypeZeroInvariantsGaloisModN}] \label{BigC} 
	If $T$ is an algebraic torus, then $\Inv^{0}_{\hom}(T, K^{M}_{1}/n) \cong H^{0}(F, T_{\sep}^{\ast}/(T_{\sep}^{\ast})^{n})$.  
\end{theorem}

The results we have obtained above follow a rich history of work on cohomological invariants: here are a few related recent examples. In \cite{Totaro}, Totaro computed all mod $p$ cohomological invariants for many important affine group schemes in characteristic $p$; in particular, under the assumption that $\chara(F) = p > 0$, Totaro independently computed $\Inv^{1}(G, K^{M}_{1}/p)$ for \emph{any} affine group scheme (\cite[Theorem 12.2]{Totaro}). The computation of invariants for smooth linear algebraic groups with values in $H^{2}(-, \QQ/\ZZ(1))$ was carried out by Alexandre Lourdeaux in \cite{Lourdeaux}.

\subsection{Acknowledgments} I would like to express my gratitude to my advisor, Alexander Merkurjev, for his advice, encouragement, and many helpful meetings. I am also grateful to Bar Roytman, David Hemminger, Will Baker, and Burt Totaro for helpful conversations. 

\subsection{Notation and Conventions} Throughout, $F$ denotes a fixed base field of arbitrary characteristic, and $F_{\sep}$ denotes a fixed separable closure. We put $\Gamma = \Gal(F_{\sep}/F)$. If $G$ is a group scheme over $F$, we write $G_{\sep}$ to denote the base change of $G$ to $F_{\sep}$, and $G^{\ast}$ to denote the character group of $G$. For an abelian group $A$ and a positive integer $n$, we write $A[n]$ to denote the subgroup of $n$-torsion elements of $A$. All group schemes are affine unless otherwise indicated. For any group scheme $G$ over $F$ and any $F$-algebra $R$, we write $\varepsilon_{R}$ to denote the identity element of $G(R)$. If $\varphi \colon Q \to Q'$ is a morphism of commutative $F$-group schemes, we write $Q^{\varphi}$ to denote the image of the embedding $Q \to Q \times Q'$ induced by $\Id_{Q}$ and the composition of $\varphi$ with the inversion map $Q' \to Q'$. For an $F$-scheme $X$, we write $\Tors_{\ast}(\varphi)(X)$ to denote the morphism $\Tors(X, Q) \to \Tors(X, Q')$ induced by $\varphi$. Likewise, if $f \colon Y \to X$ is a morphism of $F$ schemes, we write $\Tors^{\ast}(f)(Q)$ to denote the pullback morphism $\Tors(X, Q) \to \Tors(Y, Q)$.

\section{An Outline of the Argument} \label{Outline}

In this section, we give a structural overview of our argument.

\subsection{Resolution by Tori} \label{Versal}

Recall that a group scheme $G$ over $F$ is said to be {\bf diagonalizable} if the natural embedding $G^{\ast} \to F[G]^{\times}$ induces an isomorphism of Hopf $F$-algebras $F\langle G^{\ast} \rangle \to F[G]$, where $F\langle G^{\ast} \rangle$ denotes the group algebra of $G^{\ast}$ over $F$. As noted in the introduction, a group scheme $G$ over $F$ is a {\bf group of multiplicative type} if $G_{\sep}$ is diagonalizable over $F_{\sep}$. The functors 
\[G \longmapsto G_{\sep}^{\ast}, \qquad M \longmapsto (F_{\sep}\langle M \rangle)^{\Gamma}\]
define a short exact sequence-preserving equivalence between the category of (algebraic) groups of multiplicative type over $F$ and the category of (finitely generated) $\Gamma$-modules (\cite[Theorem 12.23]{MilneAG}). Under this equivalence, the full subcategory of diagonlizable $F$-group schemes is equivalent to the subcategory of $\Gamma$-modules with trivial $\Gamma$-action. 

When $G$ is an algebraic group of multiplicative type, $G$ may be embedded in a quasisplit torus $P$ such that every $G_{L}$ torsor over a field $L/F$ is the pullback of the $G$-torsor $P \to P/G$ along an $L$-point of $P/G$. Indeed, since $G_{\sep}^{\ast}$ is finitely generated, it admits a surjective morphism of $\Gamma$-modules $W \to G_{\sep}^{\ast}$ from a permutation $\Gamma$-module $W$. If $S$ denotes the kernel of this map, then let $P, T$ be the groups of multiplicative type respectively associated to $W, S$. Note that $P$ is a quasisplit torus, $T$ is a torus, and the exact sequence 
\[1 \to S \to W \to G_{\sep}^{\ast} \to 1\]
of $\Gamma$-modules yields an exact sequence
\begin{equation} \label{ResolutionByTori}  
1 \to G \xrightarrow{~f~} P \xrightarrow{~g~} T \to 1
\end{equation}
of $F$-group schemes. We will call such an exact sequence \ref{ResolutionByTori} a {\bf resolution of $G$ by tori}.

For every field extension $L/F$, the exact sequence on points $1 \to G(L) \to P(L) \to T(L)$ may be continued as follows. Let $\rho(L) \colon T(L) \to \Tors(L, G_{L})$ be the group homomorphism sending a point $\alpha \in T(L)$ to the pullback of the $G$-torsor $P \to T$ along $\alpha$. One may check that the sequence
\begin{equation} \label{ExactSeqTors} 
1 \to G(L) \xrightarrow{~f(L)~} P(L) \xrightarrow{~g(L)~} T(L) \xrightarrow{~\rho(L)~} \Tors(L, G_{L}) \xrightarrow{~\Tors_{\ast}(f_{L})~} \Tors(L, P_{L})
\end{equation}
is exact; we note that this does not depend on the fact that $G, P, T$ are of multiplicative type, and can be proven for any exact sequence of commutative group schemes. Since $P_{L}$ is a quasisplit torus, every $P_{L}$-torsor over $\Spec(L)$ is trivial. Therefore, the map $\rho(L) \colon T(L) \to \Tors(L, G_{L})$ is surjective.

The surjectivity of $\rho(L)$ allows us to relate type-one invariants for $G$ to type-zero invariants for tori, which are well understood for certain functors $H$. As $L$ varies over all field extensions of $F$, the morphisms $\rho(L)$ define a morphism of functors $\rho \colon T \to \Tors(-, G)$, which gives rise to a map $\Inv(\rho, H) \colon \Inv^{1}_{\hom}(G, H) \to \Inv^{0}_{\hom}(T, H)$ given by composition with $\rho$. Likewise, the group homomorphism $g \colon P \to T$ is a natural transformation of group-valued functors, and so induces a map $\Inv(g, H) \colon \Inv^{0}_{\hom}(T, H) \to \Inv^{0}_{\hom}(P, H)$ given by composition with $g$. The exactness of \ref{ExactSeqTors} shows that the resulting sequence
\begin{equation} \label{ExactSeqInvariants} 
1 \to \Inv^{1}_{\hom}(G, H) \xrightarrow{~\Inv(\rho, H)~} \Inv^{0}_{\hom}(T, H) \xrightarrow{~\Inv(g, H)~} \Inv^{0}_{\hom}(P, H)
\end{equation}
is exact. To describe $\Inv^{1}_{\hom}(G, H)$, it therefore suffices to determine the image of $\Inv(\rho, H)$ in $\Inv^{0}_{\hom}(T, H)$. 

\subsection{The Argument} \label{Pushforward}

Fix a positive integer $n$, let $G, P, T$ be as in the exact sequence \ref{ResolutionByTori}, and let $H = K^{M}_{1}/n$. For any group scheme $Q$ over $F$, we say that a class $V \in \Tors(Q, G)$ is {\bf normalized} if the pullback of $V$ along $\varepsilon_{F}\in Q(F)$ represents the trivial class in $\Tors(F, G)$. Let $\Tors_{\nm}(Q, G)$ denote the subgroup of normalized $G$-torsors over $Q$. 

Consider the map $\upsilon_{n}(G) \colon G^{\ast}[n] \to \Tors_{\nm}(T, \bmu_{n, F})$ which sends a character $\chi \in G^{\ast}[n]$ to $(\Tors_{\ast}(\chi)(T))(P \to T)$. We note that $\upsilon_{n}(G)$ is a group homomorphism. Indeed, for any $F$-scheme $X$, the map $\Tors(X, G) \times \Tors(X, G) \to \Tors(X, G \times G)$ sending a pair of representatives $E_{1} \to X, E_{2} \to X$ to the universal map $E_{1} \times E_{2} \to X$ is a group isomorphism. If $\Delta_{G} \colon G \to G \times G$ denotes the diagonal map, and $m_{G} \colon G \times G \to G$ denotes the group multiplication, then up to the preceding identification, $\Tors_{\ast}(m_{G})(X)$ is the group operation, and $\Tors_{\ast}(\Delta_{G})(X)$ is the diagonal embedding. Hence, if $\chi, \chi' \in G^{\ast}[n]$, then we have $\Tors_{\ast}(\chi\chi')(X) = \Tors_{\ast}(\chi)(X) + \Tors_{\ast}(\chi')(X)$, since $\chi\chi'$ factors as $m_{\bmu_{n, F}} \circ (\chi \times \chi') \circ \Delta_{G}$. This argument also explains why $\Phi(G, n)$ is a group homomorphism. 

Suppose we were armed with the following facts:

\begin{enumerate}[(1)]
\item The sequence
\begin{equation} \label{CharactersToTorsors} 
1 \to G^{\ast}[n] \xrightarrow{~\upsilon_{n}(G)~} \Tors_{\nm}(T, \bmu_{n, F}) \xrightarrow{~\Tors^{\ast}(g)(\bmu_{n, F})~} \Tors_{\nm}(P, \bmu_{n, F})
\end{equation}
is exact. 
\item For any smooth, connected, reductive group $R$ over $F$, there is a group isomorphism $\tilde{\Lambda}_{n}(R) \colon \Tors_{\nm}(R, \bmu_{n, F}) \to \Inv^{0}_{\hom}(R, K^{M}_{1}/n)$. 
\item The diagram
\begin{center}
	\begin{tikzcd}
	{G^{\ast}[n]} \arrow[rr, "\upsilon_{n}(G)"] \arrow[dd, "{\Phi(G, n)}"] &  & {\Tors_{\nm}(T, \bmu_{n, F})} \arrow[rr, "{\Tors^{\ast}(g)(\bmu_{n, F})}"] \arrow[dd, "\tilde{\Lambda}_{n}(T)"] &  & {\Tors_{\nm}(P, \bmu_{n, F})} \arrow[dd, "\tilde{\Lambda}_{n}(P)"] \\
	&  &                                                                        &  &                                                                                                                 &  &                                                                    \\
	{\Inv^{1}_{\hom}(G, K^{M}_{1}/n)} \arrow[rr, "{\Inv(\rho, K^{M}_{1}/n)}"']            &  & {\Inv^{0}_{\hom}(T, K^{M}_{1}/n)} \arrow[rr, "{\Inv(g, K^{M}_{1}/n)} "']                                                        &  & {\Inv^{0}_{\hom}(P, K^{M}_{1}/n)}                                 
	\end{tikzcd}
\end{center}
commutes. 
\end{enumerate}

If these three statements hold, then an easy diagram chase using the exactness of \ref{CharactersToTorsors} and \ref{ExactSeqInvariants} shows that $\Phi(G, n)$ is an isomorphism. The remainder of this paper is dedicated to proving these three facts, and carefully explaining why the induced map $\Phi(G)$ is an isomorphism. The remaining sections are organized as follows.

Section \ref{MuNTorsors} gives a thorough treatment of $\bmu_{n, F}$-torsors, laying the ground work for facts $(1)$ and $(2)$. We will provide a Galois theoretic-interpretation of the group $\Tors_{\nm}(T, \bmu_{n, F})$ which will allow us to interpret sequence \ref{CharactersToTorsors} as an exact sequence arising in Galois cohomology in section \ref{TypeOneInvariants}. We will also prove a pullback formula for $\bmu_{n, F}$-torsors over a smooth, connected, reductive group $R$ which shows that normalized $\bmu_{n, F}$-torsors over $R$ give rise to homomorphic type-zero invariants of $R$. 

Section \ref{TypeZeroInvariants} is devoted to constructing the map $\tilde{\Lambda}_{n}(R)$ for any smooth, connected, reductive group $R$, and proving it is a group isomorphism. We also give a description of type-zero invariants for $R$ with values in $K^{M}_{1} \otimes_{\ZZ} \QQ/\ZZ$.

The final section (\ref{TypeOneInvariants}) will prove facts (1) and (3), yielding Theorem \ref{BigA}. As noted, we will then deduce Theorem \ref{BigB} from Theorem \ref{BigA} via a detailed examination of $\Phi(G)$.

\section{$\mu_{n, F}$-Torsors} \label{MuNTorsors}

Throughout this section, let $n$ denote a fixed positive integer. As indicated in the previous section, an essential ingredient in the proof of Theorem \ref{TypeOneGMTModN} is a robust understanding $\bmu_{n, F}$-torsors over an $F$-scheme $X$. In this section, we recall several well-known characterizations of $\bmu_{n, F}$-torsors. Our main results are Theorems \ref{TorsKerDivMap}, \ref{PointsMultTorsors} and \ref{GalFixedTorsors}. Theorem \ref{TorsKerDivMap} explains that when $G$ is a smooth, connnected group, $\Tors(G, \bmu_{n, F})$ may be identified with the kernel of the divisor map $\partial_{n}(G) \colon F(G)^{\times}/(F(G)^{\times})^{n} \to \Div(G)/n\Div(G)$. Under the further assumption that $G$ is reductive, Theorem \ref{PointsMultTorsors} proves a formula relating the pullbacks of a class in $\Tors(G, \bmu_{n, F})$ along points $\alpha, \beta \in G(M)$ to its pullback along the product $\alpha\beta \in G(M)$, where $M$ is a field extension of $F$. Theorem \ref{GalFixedTorsors} computes the Galois fixed points of $\Tors_{\nm}(G_{\sep}, \bmu_{n, F_{\sep}})$ when $G$ is geometrically integral, $G_{\sep}^{\ast}$ is torsion-free, and $G_{\sep}$ has trivial divisor class group.

\subsection{The Group $\Psi(A, n)$} \label{BundlesAreTorsors}

\begin{defn}
	For any commutative ring $A$, let $\Psi(A, n)$ denote the set of equivalence classes of pairs $(\mathcal{L}, \varphi)$, where $\mathcal{L} \in  \Pic(A)[n]$, $\varphi$ is an $A$-module isomorphism $\mathcal{L}^{\otimes n} \to A$, and two pairs $(\mathcal{L}, \varphi), (\mathcal{L}', \varphi')$ are equivalent if and only if there is an isomorphism of $A$-modules $\rho \colon \mathcal{L} \to \mathcal{L}'$ such that $\varphi' \circ \rho^{\otimes n} = \varphi$. 
\end{defn}

We record the following observations about $\Psi(A, n)$, which are straightforward to check:

\begin{enumerate}[(1)]
	\item The tensor product induces a group operation on $\Psi(A, n)$: one defines the product of classes $[(\mathcal{L}, \varphi)], [(\mathcal{L}', \varphi')] \in \Psi(A, n)$ to be $[(\mathcal{L} \otimes_{A} \mathcal{L}', \varphi \otimes_{A} \varphi')]$, where $\varphi \otimes_{A} \varphi'$ really refers to the composition
	\[(\mathcal{L} \otimes_{A} \mathcal{L}')^{\otimes n} \xrightarrow{~\sim~} (\mathcal{L}^{\otimes n}) \otimes_{A} (\mathcal{L}')^{\otimes n} \xrightarrow{~\varphi \otimes_{A} \varphi'~} A \otimes_{A} A \xrightarrow{~\sim~} A.\]
	The identity class is represented by the pair $(A, \Id_{A})$, and the inverse of a class $[(\mathcal{L}, \varphi)]$ is given by $[(\mathcal{L}^{\ast}, (\varphi^{-1})^{\ast})]$, where $\mathcal{L}^{\ast}$ is the dual bundle to $\mathcal{L}$, and $(\varphi^{-1})^{\ast}$ is the composition
	\[(\mathcal{L}^{\ast})^{\otimes n} \xrightarrow{~\sim~} (\mathcal{L}^{\otimes n})^{\ast} \xrightarrow{~(\varphi^{-1})^{\ast}~} A^{\ast} \xrightarrow{~\sim~} A. \]
	\item For any ring morphism $f \colon A \to B$, extension of scalars induces a group morphism $\Psi(-, n)(f) \colon \Psi(A, n) \to \Psi(B, n)$ sending $[(\mathcal{L}, \varphi)]$ to $[(\mathcal{L} \otimes_{A} B, \varphi \otimes_{A} \Id_{B})]$, where $\varphi \otimes_{A} \Id_{B}$ really denotes the composition
	\[(\mathcal{L} \otimes_{A} B)^{\otimes n} \xrightarrow{~\sim~} \mathcal{L}^{\otimes n} \otimes_{A} B \xrightarrow{\varphi \otimes_{A} \Id_{B}} A \otimes_{A} B \xrightarrow{~\sim~} B.\]
	In this way, the association $A \mapsto \Psi(A, n)$ defines a functor $\Psi(-, n)$ from $\Rings$ to $\AbGrps$.
	\item For any positive integer $m$ with $n$ dividing $m$, there is a morphism of functors $\omega_{n, m} \colon \Psi(-, n) \to \Psi(-, m)$ defined for a commutative ring $A$ by $\omega_{n, m}(A)[(\mathcal{L}, \varphi)] = [(\mathcal{L}, \varphi^{\otimes m/n})]$, where by $\varphi^{m/n}$ we mean the composition of isomorphisms
	\[\mathcal{L}^{\otimes m} \xrightarrow{~\sim~} (\mathcal{L}^{\otimes n})^{\otimes m/n}\xrightarrow{~\varphi^{\otimes m/n}~} A^{\otimes m/n} \xrightarrow{~\sim~} A.\] 
\end{enumerate}

There is a convenient way to produce elements of $\Psi(A, n)$ which can be described as follows. Fix an element $y \in A^{\times}$, and consider the $A$-algebra $R_{y} := A[X]/\langle X^{n} - y \rangle$; we denote the residue class of $X$ in $R_{y}$ by $y^{1/n}$. If $\mathcal{L}_{y}$ denotes the free $A$-submodule of $R_{y}$ generated by $y^{1/n}$, one immediately sees that the ``multiplication'' map $\varphi_{y} \colon \mathcal{L}_{y}^{\otimes n} \to A$ sending $x_{1}y^{1/n} \otimes \cdots \otimes x_{n}y^{1/n}$ to $yx_{1}x_{2} \cdots x_{n}$ is an isomorphism of $A$-modules, and the pair $(\mathcal{L}_{y}, \varphi_{y})$ represents a class in $\Psi(A, n)$. 

One readily checks that the map $A^{\times} \to \Psi(A, n)$ sending $y \in A^{\times}$ to $[(\mathcal{L}_{y}, \varphi_{y})]$ is a group homomorphism whose kernel is exactly $(A^{\times})^{n}$, and so we obtain a well-defined injective group morphism $\Delta_{n}(A) \colon A^{\times}/(A^{\times})^{n} \to \Psi(A, n)$. Moreover, this collection of maps is functorial in $A$: in other words, if $\mathcal{K}^{n}$ denotes the functor from $\Rings$ to $\AbGrps$ sending a commutative ring $A$ to $A^{\times}/(A^{\times})^{n}$, the collection of maps $\Delta_{n}(A)$ as $A$ varies defines a natural transformation $\Delta_{n} \colon \mathcal{K}^{n} \to \Psi(-, n)$. 

On the other hand, for any commutative ring $A$, there is a well-defined surjective group homomorphism $\Theta_{n}(A) \colon \Psi(A, n) \to \Pic(A)[n]$ which sends a class $[(\mathcal{L}, \varphi)] \in \Psi(A, n)$ to $[\mathcal{L}]$, and the collection of such $\Theta_{n}(A)$ as $A$ varies likewise determines a natural transformation $\Theta_{n} \colon \Psi(-, n) \to \Pic(-)[n]$. The relationship between $\Delta_{n}$ and $\Theta_{n}$ is explained by the following proposition. 

\begin{prop} \label{MuTorsExactSeq}  
	For any commutative ring $A$, the sequence
	\[1 \to A^{\times}/(A^{\times})^{n} \xrightarrow{~\Delta_{n}(A)~} \Psi(A, n) \xrightarrow{~\Theta_{n}(A)~} \Pic(A)[n] \to 0\]
	is exact. 
\end{prop}

\begin{proof}
	The inclusion $\Imm(\Delta_{n}(A)) \subset \ker(\Theta_{n}(A))$ is immediate, since $\mathcal{L}_{y}$ is a free $A$-module for any $y \in A^{\times}$ by construction. Suppose that $(\mathcal{L}, \varphi) \in \ker(\Theta_{n}(A))$, i.e. that $\mathcal{L}$ is free. Let $\psi \colon \mathcal{L} \to A$ be an isomorphism of $A$-modules. Consider the composition of isomorphisms
	\[A \xrightarrow{~\varphi^{-1}~} \mathcal{L}^{\otimes n} \xrightarrow{~\psi^{\otimes n}} A^{\otimes n} \xrightarrow{~\sim~} A. \]
	Every $A$-module isomorphism $A \to A$ is given by multiplication by some invertible element of $A$, so the composition above is multiplication by $x$ for some $x \in A^{\times}$. Put $y = x^{-1}$, and let $\alpha \colon A \to \mathcal{L}_{y}$ be the isomorphism sending $a$ to $ay^{1/n}$. Then one easily checks that $\alpha \circ \psi$ is an isomorphism between $(\mathcal{L}, \varphi)$ and $(\mathcal{L}_{y}, \varphi_{y})$. 
\end{proof} 

\begin{cor}\label{KMuNTorsors}
	If $\Pic(A)[n] = 0$, then $\Delta_{n}(A)$ is an isomorphism. $\qed$
\end{cor}

Suppose now that $A$ is an $F$-algebra. To any element $[(\mathcal{L}, \varphi)]$ of $\Psi(A, n)$, one may associate a $\ZZ/n\ZZ$-graded $A$-algebra $\Tw(\mathcal{L}, \varphi)$. As an $A$-module, we set 
\[\Tw(\mathcal{L}, \varphi) := A \oplus \mathcal{L} \oplus \mathcal{L}^{\otimes 2} \oplus \cdots \oplus \mathcal{L}^{\otimes n-1}.\] 
The multiplicative structure on $\Tw(\mathcal{L}, \varphi)$ is induced by the isomorphisms $\mathcal{L}^{\otimes i} \otimes_{A} \mathcal{L}^{\otimes j} \to \mathcal{L}^{\otimes i+j}$ for $i+j < n$, and $\mathcal{L}^{\otimes i} \otimes_{A} \mathcal{L}^{\otimes j} \to \mathcal{L}^{\otimes n} \otimes_{A} \mathcal{L}^{\otimes (n-(i+j))} \xrightarrow{\varphi \otimes \Id} A \otimes_{A} \mathcal{L}^{\otimes (n-(i+j))} \to \mathcal{L}^{n-(i+j)}$ for $i+j \geqslant n$. Note that the inclusion morphism $A \to \Tw(\mathcal{L}, \varphi)$ is faithfully flat, because $\Tw(\mathcal{L}, \varphi)$ is finitely generated and projective as an $A$-module. In fact, the dual morphism $\Spec(\Tw(\mathcal{L}, \varphi)) \to \Spec(A)$ is a $\bmu_{n, F}$-torsor over $\Spec(A)$, and we can say yet more, as the next theorem explains. Let $\lambda_{n}(A) \colon \Psi(A, n) \to \Tors(\Spec(A), \bmu_{n, F})$ be the set map sending $[(\mathcal{L}, \varphi)]$ to the $\bmu_{n, F}$-torsor class represented by the map $\Spec(\Tw(\mathcal{L}, \varphi)) \to \Spec(A)$.

\begin{thm} \label{PsiTorsors}
	The map $\lambda_{n}(A)$ is a well-defined {\bf group} isomorphism. Moreover, as $A$ varies over all $F$-algebras, the collection of maps $\lambda_{n}(A)$ defines a natural isomorphism $\lambda_{n} \colon \Psi(-, n) \to \Tors(-, \bmu_{n, F})$.   
\end{thm}

\begin{proof}
	See \cite[\href{https://stacks.math.columbia.edu/tag/03PK}{Tag 03PK}]{Stacks}. Alternatively, see \cite[page 125]{EC}. 
\end{proof}

We note that for any $y \in A^{\times}$, the universal map $A[X]/\langle X^{n} -y \rangle \to \Tw(\mathcal{L}_{y}, \varphi_{y})$ sending $\overline{X}$ to $y^{1/n} \in \mathcal{L}_{y}$ is an isomorphism of $(\ZZ/n\ZZ)$-graded $A$-algebras. Hence, the composition $\lambda_{n}(A) \circ \Delta_{n}(A)$ takes $y \in A^{\times}$ to the class of the $\bmu_{n, F}$-torsor $\Spec(A[X]/\langle X^{n}- y \rangle) \to \Spec(A)$. We put $\mathbf{ \Sigma_{n} := \lambda_{n} \circ \Delta_{n}}$. 

\subsection{Divisors}

When $A$ is a domain, there is another description of $\Psi(A, n)$ in terms of divisors. Let $K$ denote the field of fractions of $A$, and let $\Cart(A)$ denote the group of invertible fractional ideals of $A$. Likewise, if $A$ is a Krull domain, let $\Div(A)$ be the free abelian group generated by the codimension $1$ points of $\Spec(A)$. We write $\dv(A) \colon \Cart(A) \to \Div(A)$ to denote the usual valuation homomorphism which sends a fractional ideal $I$ to the formal sum of its valuations at each height one prime of $A$. We let $\partial(A) \colon K^{\times} \to \Div(A)$ denote the group morphism sending $x \in K^{\times}$ to $\dv(xA)$.   

Consider the set $C(A, n)$ consisting of pairs $(I, f)$ where $I \in \Cart(A)$, and $f \in K^{\times}$ such that $I^{n} = fA$. The binary operation on $C(A, n)$ defined by $(I, f) \cdot (I', f') = (II', ff')$ gives $C(A, n)$ the structure of a group with identity element $(A, 1)$. There is a group homomorphism $K^{\times} \to C(A, n)$ sending $x \in K^{\times}$ to $(xA, x^{n})$, and we set $\Cart(A, n)$ to be the cokernel of this morphism. 

If we further assume that $A$ is a Krull domain, then there is an analagous construction $\Div(A, n)$. If $D(A, n)$ denotes the set of pairs $(D, g)$ where $D \in \Div(A)$ and $g \in K^{\times}$ such that $\partial(A)(g) = nD$, then $\Div(A, n)$ is defined to be the cokernel of group homomorphism $K^{\times} \to D(A, n)$ which sends $x \in K^{\times}$ to the pair $(\partial(A)(x), x^{n})$. One may check that the map $\dv(A) \colon \Cart(A) \to \Div(A)$ described above descends to a group morphism $\dv_{n}(A) \colon \Cart(A, n) \to \Div(A, n)$, and this map is an isomorphism if $A$ is regular. 
	
For any element $I \in \Cart(A)$, the multiplication map $I^{\otimes n} \to I^{n}$ is an isomorphism of $A$-modules, since $I$ is projective of rank $1$. Given a pair $(I, f)$ in $C(A, n)$, we may produce a pair $(I, m_{f})$ which represents a class in $\Psi(A, n)$, where $m_{f}$ is the composition of $A$-isomorphisms $I^{\otimes n} \xrightarrow{~\sim~} I^{n} \xrightarrow{~\cdot f^{-1}~} A$. One may check that the resulting set map $\Omega(A) \colon C(A, n) \to \Psi(A, n)$ sending a pair $(I, f)$ to $[(I, m_{f})]$ is a group homomorphism.

\begin{prop}
	The morphism $\Omega(A) \colon C(A, n) \to \Psi(A, n)$ is surjective, and the kernel is precisely the image of the group morphism $K^{\times} \to C(A, n)$ sending $x \in K^{\times}$ to $(xA, x^{n})$. Therefore, $\Omega(A)$ descends to a well-defined group isomorphism $\Omega_{n}(A) \colon \Cart(A, n) \to \Psi(A, n)$. 
\end{prop}

\begin{proof}
	Let $m \colon A^{\otimes n} \to A$ denote the multiplication map. If $(I, f) \in C(A, n)$ belongs to $\ker(\Omega(A))$, then there is an isomorphism $\rho \colon A \to I$ such that $m_{f} \circ \rho^{\otimes n} = m$. Then $I$ is principal, generated by $\rho(1) =: x \in K^{\times}$, and 
	\[1 = m(1 \otimes \cdots \otimes 1) = m_{f}(x \otimes \cdots \otimes x) = x^{n}/f,\]
	so $x^{n} = f$, and $(I, f) = (xA, x^{n})$. On the other hand, for any $x \in K^{\times}$, the class $[(xA, m_{x^{n}})]$ in $\Psi(A, n)$ is trivial, via the isomorphism $A \to xA$ sending $1$ to $x$. \\  
	Now, let the pair $(\mathcal{L}, \varphi)$ represent a class in $\Psi(A, n)$. Let $I \subset K$ denote the image of $\mathcal{L}$ under the composition of the $A$-embedding $\mathcal{L} \to \mathcal{L} \otimes_{A} K$ with a fixed $K$-module isomorphism $\mathcal{L} \otimes_{A} K \to K$. After clearing denominators, we may assume that $I \subset A \subset K$, so that $I$ is an ideal of $A$. Since $\mathcal{L}$ is projective of rank $1$, $I$ is an invertible ideal of $A$. \\
	Let $\alpha \colon \mathcal{L} \to I$ denote our $A$-module isomorphism of $\mathcal{L}$ onto $I$. If $f$ denotes the image of $1$ under the sequence of isomorphisms $A \xrightarrow{~\varphi^{-1}~} \mathcal{L}^{\otimes n} \xrightarrow{~\alpha^{\otimes n}~} I^{\otimes n} \xrightarrow{~\sim~} I^{n}$, then one sees that $I^{n} = fA$, and $\alpha$ is an isomorphism between $(\mathcal{L}, \varphi)$ and $(I, m_{f})$.    
\end{proof}

The above proof shows that every element of $\Psi(A, n)$ admits a representative of the form $(I, m_{f})$ where $I \subset A$ is an invertible fractional ideal of $A$ satisfying $I^{n} = fA$ for some nonzero $f \in A$. We will call such a representative an {\bf ideal representative} of a class in $\Psi(A, n)$. 

\begin{cor}\label{NonzeroIdealTors}
	Let $A$ be a normal domain with field of fractions $K$, and let $X$ be a class in $\Psi(A, n)$. Let $M$ be a domain, and let $\alpha_{1}, \ldots, \alpha_{n} \colon A \to M$ be ring morphisms. Then one can choose an ideal representative $(\tilde{I}, m_{\tilde{f}})$ for $X$ such that $\alpha_{i}(\tilde{f}) \in M\setminus \{0\}$ for each $1 \leqslant i \leqslant n$. 
\end{cor}

\begin{proof}
	For each $1 \leqslant i \leqslant n$, put $\mfp_{i} = \ker(\alpha_{i}) \in \Spec(A)$. Let $S$ be the multiplicative subset of $A$ defined by $S = A\setminus\bigcup_{i=1}^{n} \mfp_{i}$; then $B := S^{-1}A$ is a semi-local ring whose maximal ideals are a subset of $\{\mfp_{i}B\}_{i = 1}^{n}$. Let $(I, m_{f})$ be an ideal representative for $X$, and put $J = S^{-1}I$. Since $J$ is a $B$-module of constant rank $1$ and $B$ is semi-local, $J$ is a free $B$-module of rank $1$, hence principal. Say $J$ is generated by $0 \neq y/z \in B$. Since $I^{n} = fA$, we have $J^{n} = fB$, whence
	\[\left(\frac{y}{z}\right)^{n} = f \cdot u \]
	for some unit $u \in B^{\times}$. If $u = v/w$ for $v \in A, w \in S$, we must have $v \in S$ as well. Put $g = vz/y \in K^{\times}$, so that
	\[v^{n-1}w = g^{n}f, \]
	and let $\tilde{I} = gI$, $\tilde{f} = v^{n-1}w \in S \subset A$. Then $\tilde{I}^{n} = g^{n}(I^{n}) = g^{n}fA = v^{n-1}wA$, and the map $I \to \tilde{I}$ given by multiplication by $g$ is an isomorphism between $(I, m_{f})$ and $(\tilde{I}, m_{\tilde{f}})$ in $\Psi(A, n)$. Moreover, $\tilde{f} \in S$, and so $\alpha_{i}(\tilde{f}) \in M\setminus \{0\}$ for each $i$. It remains to show that $\tilde{I} \subset A$. Let $x \in I$; then $(gx)^{n} \in \tilde{I}^{n} \subset A$. But then $gx$ is a root of $X^{n}-(gx)^{n} \in A[X]$, and so is integral over $A$, and therefore belongs to $A$.      
\end{proof}

Notice that if $M$ is a field, then $\Pic(M)$ is trivial, so $\Delta_{n}(M)$ is an isomorphism by Corollary \ref{KMuNTorsors}.

\begin{prop}\label{FieldMorphismTorsImage}
	Let $A$ be a normal domain, and let $M$ be a field. Let $\alpha \colon A \to M$ be a ring morphism. Let $X \in \Psi(A, n)$, and let $(I, m_{f}) \in \Psi(A, n)$ be an ideal representative for $X$ satisfying $\alpha(f) \neq 0$. Then $(\Delta_{n}(M)^{-1} \circ \Psi(-, n)(\alpha))(X) = [\alpha(f)^{-1}].$ 	
\end{prop}

\begin{proof}
	Consider the morphism of $M$-vector spaces $\tau \colon I \otimes_{A} M \to \mathcal{L}_{\alpha(f)^{-1}}$ given on simple tensors by $\tau(x \otimes z) = \alpha(x)z\alpha(f)^{-1/n}$. Since $I \otimes_{A}M$ and $L_{\alpha(f)^{-1}}$ are both $M$-vector spaces of dimension $1$, the map $\tau$ is an isomorphism provided it is nonzero. Indeed, this is the case, since $f \in I$, and so $\tau(f \otimes 1) = \alpha(f)\alpha(f)^{-1/n}$ is nonzero. It is straightforward to check that $\tau$ is an isomorphism between $\Psi(-, n)(\alpha)(X)$ and $\mathcal{L}_{\alpha(f)^{-1}}$. 
\end{proof}

\begin{cor} \label{IdealRepDelta}
	Let $A$ be a normal domain with field of fractions $K$, let $M$ be a field, and let $\alpha \colon A \to M$ be a morphism of rings. Let $x \in A^{\times}$, and let $(I, m_{f})$ be an ideal representative for $\Delta_{n}(A)(x)$. Then there is a nonzero element $y \in A$ such that $I = yA$ and $y^{n}/f = x$ in $K$, and $[\alpha(x)] = [\alpha(f)^{-1}]$ in $M^{\times}/(M^{\times})^{n}$. We deduce $\mathcal{K}^{n}(\alpha) = \Delta_{n}(M)^{-1} \circ \Psi(-, n)(\alpha) \circ \Delta_{n}(A)$. 
\end{cor}
\begin{proof}
	Since $[(I, m_{f})]$ and $[(\mathcal{L}_{x}, \varphi_{x})]$ are equal as classes in $\Psi(A, n)$, there is an isomorphism of $A$-modules $\omega \colon \mathcal{L}_{x} \to I$ such that $m_{f} \circ \omega^{\otimes n} = \varphi_{x}$. As $\mathcal{L}_{x}$ is free, $I$ is a (nonzero) principal ideal, generated by $y := \omega(1 \cdot x^{1/n}) \in A$. We thus have
	\[x = \varphi_{x}(x^{1/n} \otimes \cdots \otimes x^{1/n}) = m_{f}(\omega^{n}(x^{1/n} \otimes \cdots \otimes x^{1/n})) = m_{f}(y \otimes \cdots \otimes y) = y^{n}/f \]
	as claimed. Moreover, since $xf = y^{n}$ and $\alpha(x), \alpha(f) \in M^{\times}$, this forces $\alpha(y) \in M^{\times}$, and so $[\alpha(x)] \cdot [\alpha(f)] = [\alpha(y)^{n}] = [1] \in M^{\times}/(M^{\times})^{n}$. 
\end{proof}

Suppose $A$ is a normal domain, let $K$ be its field of fractions, and let $\xi \colon A \to K$ be the canonical localization map. Fix $X \in \Psi(A, n)$, and let $(I, m_{f})$ be an ideal representative for $X$. By Proposition \ref{FieldMorphismTorsImage}, $\Delta_{n}(K)^{-1}(\Psi(-, n)(\xi)(X)) = [\xi(f)^{-1}] = [1/f]$. If  $\partial_{n}(A) \colon K^{\times}/(K^{\times})^{n} \to \Div(A)/n\Div(A)$ denotes the map induced by $\partial(A)$, then
\[\partial_{n}(A)([1/f]) = -[\partial(A)(f)] = -[\dv(fA)] = -[n\dv(I)] = 0 \in \Div(A)/n\Div(A) \]
Hence, the map $\Delta_{n}(K)^{-1} \circ \Psi(-, n)(\xi)$ takes image in $\ker(\partial_{n}(A)) \subset K^{\times}/(K^{\times})^{n}$. If we further assume $A$ is regular, then Theorem \ref{TorsKerDivMap} shows $\Delta_{n}(K)^{-1} \circ \Psi(-, n)(\xi)$ (viewed by abuse of notation as a map $\Psi(A, n) \to \ker(\partial_{n}(A))$) is an isomorphism. 

\begin{thm} \label{TorsKerDivMap}
	Let $A$ be a regular domain, let $K$ be its field of fractions, and let $\xi \colon A \to K$ be the canonical localization map. Then the map $\Delta_{n}(K)^{-1} \circ \Psi(-, n)(\xi) \colon \Psi(A, n) \to \ker(\partial_{n}(A))$ is an isomorphism. 
\end{thm}
\begin{proof}
	First, consider the morphism $\zeta(A) \colon \ker(\partial_{n}(A)) \to \Cl(A)[n]$ defined as follows: if $[x] \in K^{\times}/(K^{\times})^{n}$ belongs to the kernel of $\partial_{n}(A)$, then $\partial(A)(x) \in n\Div(A)$. Define $\zeta(A)$ by sending $[x]$ to $[D]$, where $D$ satisfies $nD = \partial(A)(x)$; note that $D$ must be unique, since $\Div(A)$ is free. This map is well-defined, because if $x' = xy^{n}$ for $y \in K^{\times}$, then $\partial(A)(xy^{n}) = n(D+\partial(A)(y))$, and $[D] = [D+\partial(A)(y)]$ in $\Cl(A)$. We claim that the sequence
	\[1 \to A^{\times}/(A^{\times})^{n} \xrightarrow{~\mathcal{K}^{n}(\xi)~} \ker(\partial_{n}(A)) \xrightarrow{~\zeta(A)~} \Cl(A)[n] \to 0 \]
	is exact. Indeed, $\mathcal{K}^{n}(\xi)$ is an injection because $A$ is integrally closed in $K$. Moreover, if $[D] \in \Cl(A)[n]$, then $nD = \partial(A)(x)$ for some $x \in K^{\times}$, and so $\zeta(A)([x]) = [D]$. Hence, $\zeta(A)$ is surjective. \\
	It remains to check exactness at $\ker(\partial_{n}(A))$. Clearly, $\Imm(\mathcal{K}^{n}(\xi)) \subset \ker(\zeta(A))$, so suppose that $[x] \in \ker(\zeta(A))$. Then $\partial(A)(x) = n\partial(A)(y) = \partial(A)(y^{n})$ for some $y \in K^{\times}$, whence $x = y^{n} \cdot x'$ for some $x' \in A^{\times}$, and so $[x] = [x']$ in $K^{\times}/(K^{\times})^{n}$. \\
	Now, since $A$ is regular, $\dv(A) \colon \Cart(A) \to \Div(A)$ is an isomorphism, and so induces an isomorphism $\Pic(A)[n] \to \Cl(A)[n]$. Let $\nu(A) \colon \Pic(A)[n] \to \Cl(A)[n]$ be the composition of this isomorphism with the inversion automorphism $\Cl(A)[n] \to \Cl(A)[n]$. I claim that $\zeta(A) \circ \Delta_{n}(K)^{-1} \circ \Psi(-, n)(\xi) = \nu(A) \circ \Theta_{n}(A)$. Indeed, let $X$ be a class in $\Psi(A, n)$, and let $(I, m_{f})$ be an ideal representative for $X$. Then 
	\[(\nu(A) \circ \Theta_{n}(A))(X) = \nu(A)([I]) =  -[\dv(I)].\]
	On the other hand, 
	\[(\zeta(A) \circ \Delta_{n}(K)^{-1} \circ \Psi(-, n)(\xi))(X) = \zeta(A)([1/f])\]
	by Proposition \ref{FieldMorphismTorsImage}. But $I^{n} = fA$, so $\partial(A)(1/f) = -\dv(I^{n}) = -n\dv(I)$, and thus $\zeta(A)(1/f) = -[\dv(I)]$. By Proposition \ref{IdealRepDelta}, $\mathcal{K}^{n}(\xi) = \Delta_{n}(K)^{-1} \circ \Psi(-, n)(\xi) \circ \Delta_{n}(A)$, so we have a commutative diagram of abelian groups 
	\begin{center}
		\begin{tikzcd}[column sep = .5 in, row sep = .7 in]
		1 \arrow{r} & A^{\times}/(A^{\times})^{n} \arrow{d}{\Id_{A^{\times}/(A^{\times})^{n}}} \arrow{r}{\Delta_{n}(A)} & \arrow{d}{\Delta_{n}(K)^{-1} \circ \Psi(-, n)(\xi)} \Psi(A, n) \arrow{r}{\Theta_{n}(A)} & \arrow{d}{\nu(A)} \Pic(A)[n] \arrow{r} & 0 \\
		1 \arrow{r} & A^{\times}/(A^{\times})^{n} \arrow{r}{\mathcal{K}^{n}(\xi)} & \ker(D(A)_{n}) \arrow{r}{\zeta(A)} & \Cl(A)[n] \arrow{r} & 0 	
		\end{tikzcd}
	\end{center}
	whose rows are exact. Since $\Id_{A^{\times}/(A^{\times})^{n}}$ and $\nu(A)$ are isomorphisms, $\Delta_{n}(K)^{-1} \circ \Psi(-, n)(\xi)$ must be an isomorphism as well.
\end{proof}

\subsection{Pulling Back Torsors Along Products of Points}

\begin{defn}
	Let $G$ be an algebraic group over a field $F$, and let $A = F[G]$. Let $\varepsilon_{F} \in G(F)$ denote the identity element. We say a class $X \in \Psi(A, n)$ is {\bf normalized} if $X \in \ker(\Psi(-, n)(\varepsilon_{F}))$. We denote the subgroup of $\Psi(A, n)$ consisting of normalized elements by $\Psi_{\nm}(A, n)$. Likewise, we set $\Cart_{\nm}(A, n) = \Omega_{n}(A)^{-1}(\Psi_{\nm}(A, n)))$, and if $G$ is smooth, then we set $\Div_{\nm}(A, n) = \dv_{n}(A)(\Cart_{\nm}(A, n))$. 
\end{defn}

We note the following properties of the subgroup $\Psi_{\nm}(A, n)$:

\begin{enumerate}[(1)]
	\item By Theorem \ref{PsiTorsors}, one sees that $\Psi_{\nm}(A, n) = \lambda_{n}(A)^{-1}(\Tors_{\nm}(G, \bmu_{n, F}))$. 
	\item The assignment $A \mapsto \Psi_{\nm}(A, n)$ defines a functor from the category of Hopf $F$-algebras to $\AbGrps$. Moreover, if $M/F$ is a field extension, and $\alpha \colon A \to A_{M}$ denotes the canonical base change morphism, then the restriction of $\Psi(-, n)(\alpha)$ to $\Psi_{\nm}(A, n)$ takes image in $\Psi_{\nm}(A_{M}, n)$.   
	\item If $G$ is a smooth, connected group, then $\Delta_{n}(A)^{-1}(\Psi_{\nm}(A, n)) = G^{\ast}/(G^{\ast})^{n} \subset A^{\times}/(A^{\times})^{n}$. This follows from Rosenlicht's theorem (\cite[Theorem 3]{Rosenlicht}). 
\end{enumerate}

Normalized elements play a key role in the following situation. Let $M/F$ be a field extension, and fix a class $X \in \Psi(A, n)$. Consider the map $G(M) \to \Psi(M, n)$ which sends $\alpha \in G(M)$ to $\Psi(-, n)(\alpha)(X)$. Under what conditions is this map a group homomorphism? As the following theorem shows, this is the case precisely when $X$ is normalized, provided that $G$ is smooth, connected, and reductive.  

\begin{thm}\label{PointsMultTorsors}
	Let $G$ be a smooth, connected, reductive, algebraic group over a field $F$. Put $A = F[G]$, and let $M$ be a field extension of $F$. For any $\alpha, \beta \in G(M)$, and any class $X \in \Psi(A, n)$, we have 
	\[\Psi(-, n)(\alpha)(X) \cdot \Psi(-, n)(\beta)(X) = \Psi(-, n)(\alpha\beta)(X) \cdot \Psi(-, n)(\varepsilon_{M})(X) \]
\end{thm}

\begin{proof}
	Let $(I, m_{f})$ be an ideal representative for $X$ such that $\alpha(f), \beta(f), (\alpha\beta)(f), \varepsilon_{M}(f) \in M^{\times}$; this is possible by Corollary \ref{NonzeroIdealTors}. By Proposition \ref{FieldMorphismTorsImage}, it suffices to show that
	\[[(\alpha\beta)(f)\varepsilon_{M}(f)] = [\alpha(f)\beta(f)]\]
	as classes in $M^{\times}/(M^{\times})^{n}$. Let $B = F[G \times_{F} G] = A \otimes_{F} A$, and let $E$ be the field of fractions of $B$. Let $c, p_{1}, p_{2} \colon A \to B$ be the $F$-algebra morphisms corresponding respectively to the morphisms $G \times G \to G$ given by multiplication and projection onto each component. We note that $c, p_{1}, p_{2}$ are each flat, hence injective. \\
	By \cite[Theorems 16.56 and 21.84]{MilneAG}, any connected, reductive algebraic group over a separably closed field is rational, and so the natural map $\Pic(B) \to \Pic(A) \oplus \Pic(A)$ is an isomorphism by \cite[Lemma 6.6]{Sansuc}. Moreover, up to this identification, $\Pic(c) \colon \Pic(A) \to \Pic(B)$ is the diagonal embedding, and  $\Pic(p_{i}) \colon \Pic(A) \to \Pic(B)$ is the embedding onto the $i^{\thh}$ component. Let $J_{c} = c(I)B, J_{i} = p_{i}(I)B$; since $c, p_{1}, p_{2}$ are flat, $J_{c}, J_{1}, J_{2} \in \Cart(B)$, and we have $[J_{c}] = [J_{1}] + [J_{2}]$ as classes in $\Pic(B)$. In light of the classical exact sequence
	\begin{equation} \label{ExactSeqCart} 
	1 \to B^{\times} \to E^{\times} \to \Cart(B) \to \Pic(B) \to 0
	\end{equation}
	there exists $h \in E^{\times}$ such that $J_{c} = hB \cdot J_{1} \cdot J_{2}$. Raising each side of this equation to the $n^{\thh}$ power and using the relation $I^{n} = fA$ gives the equation $c(f)B = h^{n}B \cdot (f \otimes f)B$. Appealing again to \ref{ExactSeqCart}, there exists $b \in B^{\times}$ such that $b c(f) = h^{n} (f \otimes f)$. Let $x, y \in B$ such that $h = x/y$, so that our equation reads $b c(f) y^{n} = x^{n} (f \otimes f)$. \\
	Let $\omega \colon B \to M$ be the composition of $\alpha \otimes_{F} \beta \colon B \to M \otimes_{F} M$ and the multiplication map $M \otimes_{F} M \xrightarrow{\sim} M$. Then $\omega(c(f)) = (\alpha\beta)(f)$, and $\omega(f \otimes f) = \alpha(f)\beta(f)$, so applying $\omega$ to the equation above gives
	\[(\alpha\beta)(f) \omega(b)\omega(y)^{n} = \alpha(f)\beta(f) \omega(x)^{n}\]
	Since $M$ is a field, $\mfp := \ker(\omega)$ is a prime ideal of $B$. We know that $\alpha(f), \beta(f) \in M^{\times}$, so $f \otimes f$ belongs to $B \setminus \mfp$. Hence, $h^{n} = bc(f)/(f\otimes f) \in B_{\mfp}$. Since $B$ is regular, it follows that $B_{\mfp}$ is integrally closed in $E$, so $h^{n} \in B_{\mfp}$ implies $h \in B_{\mfp}$; in particular, we have $\omega(y) \neq 0$. This also forces $\omega(x) \neq 0$, since $(\alpha\beta)(f), \omega(b) \in M^{\times}$, so we have
	\[[(\alpha\beta)(f)\omega(b)] = [\alpha(f)\beta(f)]\]
	as classes in $M^{\times}/(M^{\times})^{n}$. It remains to show that $\omega(b)$ and $\varepsilon_{M}(f)$ belong to the same class in $M^{\times}/(M^{\times})^{n}$. By Rosenlicht's theorem (\cite[Theorem 3]{Rosenlicht}), the map $F^{\times} \oplus G^{\ast} \oplus G^{\ast} \to B^{\times}$ sending $(z, \chi, \rho)$ to $z(\chi^{\sharp}(t) \otimes \rho^{\sharp}(t) )$ is an isomorphism, so $b$ can be written as $z(g \otimes g')$ for $g, g' \in A^{\times}$ group-like elements, $z \in F^{\times}$. Then $\omega(b) = z\alpha(g)\beta(g')$, and our equation in $M^{\times}/(M^{\times})^{n}$ therefore reads
	\[[(\alpha\beta)(f) \cdot z \cdot \alpha(g)\beta(g')] = [\alpha(f)\beta(f)]\]
	Our derivation of this equation did not depend on our choice of $\alpha, \beta \in G(M)$, only on the fact that $\alpha(f), \beta(f), (\alpha\beta)(f) \in M^{\times}$. In particular, since we arranged that $\varepsilon_{M}(f) \neq 0$, we can substitute $\varepsilon_{M}$ for $\alpha$ or $\beta$ in our equation. Plugging in $\alpha = \varepsilon_{M}$ and using $\varepsilon_{M}(g) = 1$ gives $[\varepsilon_{M}(f)] = [z\beta(g')]$, and likewise, plugging in $\beta = \varepsilon_{M}$ yields $[\varepsilon_{M}(f)] = [z\alpha(g)]$. Substituting both $\alpha = \varepsilon_{M}, \beta = \varepsilon_{M}$ simultaneously gives us $[z] = [\varepsilon_{M}(f)]$, whence $[\alpha(g)] = [\beta(g')] = 1$, and so $[z\alpha(g)\beta(g')] = [\varepsilon_{M}(f)]$, completing the proof. 
\end{proof}

\begin{cor} \label{NormalizedGroupOp}
	Let $G, A, M$ be as in the statement of Theorem \ref{NormalizedGroupOp}. If $X \in \Psi_{\nm}(A, n)$, then the map $G(M) \to \Psi(M, n)$ sending $\alpha \in G(M)$ to $\Psi(-, n)(\alpha)(X)$ is a group homomorphism. \qed
\end{cor}

\subsection{The Galois Action on Torsors}

Suppose that our group $G$ is smooth and connected, and $\Pic(G_{\sep})[n] = 0$. Then putting $A = F[G]$, $\Delta_{n}(A_{\sep})$ is an isomorphism by Corollary \ref{KMuNTorsors}, and the subgroup of $\Psi_{\nm}(A_{\sep}, n)$ of $\Psi(A_{\sep}, n)$ is the image of $G_{\sep}^{\ast}/(G_{\sep}^{\ast})^{n}$. Via the embedding $\Gamma \to \Aut_{F-\alg}(A_{\sep})$, $\Gamma$ acts functorially on $A_{\sep}^{\times}/(A_{\sep}^{\times})^{n}$ and $\Psi(A_{\sep}, n)$, and the map $\Delta_{n}(A_{\sep})$ is $\Gamma$-equivariant. Since the action of $\Gamma$ on $A_{\sep}^{\times}/(A_{\sep}^{\times})^{n}$ preserves the summand $G_{\sep}^{\ast}/(G_{\sep}^{\ast})^{n} \subset A_{\sep}^{\times}/(A_{\sep}^{\times})^{n}$, this shows that the action of $\Gamma$ on $\Psi(A_{\sep}, n)$ restricts to an action on $\Psi_{\nm}(A_{\sep}, n)$. 

Throughout this section, let $\alpha \colon A \to A_{\sep}$ denote the canonical base change morphism. The associated map $\Psi(-, n)(\alpha) \colon \Psi_{\nm}(A, n) \to \Psi_{\nm}(A_{\sep}, n)$ has image in $H^{0}(F, \Psi_{\nm}(A_{\sep}, n))$. If we assume that $G$ is geometrically integral, then $\Psi(-, n)(\alpha)$ is an embedding with image $H^{0}(F, \Psi_{\nm}(A_{\sep}, n))$; this is the content of Theorem \ref{GalFixedTorsors}. First, we require a lemma. 

%

\begin{lem} \label{ClassGroupCohomology} 
	Let $G$ be a smooth, geometrically integral group variety over $F$, and let $A = F[G]$. If $\Cl(A_{\sep}) = 0$, then there is an isomorphism $Z(A) \colon H^{1}(F, G_{\sep}^{\ast}) \to \Cl(A)$.   	
\end{lem}

\begin{proof}
	Let $K_{s} = \Frac(A_{\sep})$; since $G$ is geometrically integral, $K_{s} = KF_{\sep}$. Since $\Cl(A_{\sep}) = 0$, we have an exact sequence of $\Gamma$-modules
	\[1 \to (A_{\sep})^{\times} \to (K_{s})^{\times} \xrightarrow{\partial(A_{\sep})} \Div(A_{\sep}) \to 0 \]
	and therefore obtain the following long exact sequence in Galois cohomology: 
	\[H^{0}(F, (A_{\sep})^{\times}) \to H^{0}(F, K_{s}^{\times}) \to H^{0}(F, \Div(A_{\sep})) \xrightarrow{~\delta~} H^{1}(F, (A_{\sep})^{\times}) \to H^{1}(F, K_{s}^{\times}) \to \cdots \]
	Since $\Gamma \cong \Gal(KF_{\sep}/K) = \Gal(K_{s}/K)$, we have $H^{1}(F, K_{s}^{\times})$ by Hilbert Theorem 90. As $A$ is regular and geometrically integral, $\Div(\alpha)$ embeds $\Div(A)$ onto $H^{0}(F, \Div(A_{\sep}))$. By Rosenlicht's Theorem (\cite[Theorem 3]{Rosenlicht}), the map $F_{\sep}^{\times} \oplus G_{\sep}^{\ast} \to A_{\sep}^{\times}$ sending $(z, \chi)$ to $z\chi^{\sharp}(t)$ is an isomorphism of $\Gamma$-modules. Hence, $H^{1}(F, (A_{\sep})^{\times}) \cong H^{1}(F, G_{\sep}^{\ast}) \oplus H^{1}(F, (F_{\sep})^{\times}) = H^{1}(F, G_{\sep}^{\ast})$. We thus have a commutative diagram
	\begin{center}
		\begin{tikzcd}[column sep = .4 in, row sep = .5 in]
		A^{\times} \arrow{d}{} \arrow{r} & K^{\times} \arrow{r}{\partial(A)} \arrow{d}{} & \Div(A)  \arrow{d}{\Div(\alpha)} \arrow{r}{\dv(A)} & \Cl(A) \arrow{r} & 0  \\
		H^{0}(F, (A_{\sep})^{\times}) \arrow{r} & H^{0}(F,K_{s}^{\times}) \arrow{r}{\partial(A_{\sep})} & H^{0}(F,\Div(A_{\sep})) \arrow{r}{\delta} & H^{1}(F, G_{\sep}^{\ast}) \arrow{r} & 0 
		\end{tikzcd}
	\end{center}
	with exact rows and vertical arrows isomorphisms. By the universal property of the cokernel, $\Div(\alpha)$ descends to a well-defined map $Z(A) \colon \Cl(A) \to H^{1}(F, G_{\sep}^{\ast})$ which sends the $[D] \in \Cl(A)$ to $\delta(\Div(\alpha)(D))$. By (e.g.) the Five Lemma, $Z(A)$ is an isomorphism. 
\end{proof}

Note that we can be more explicit in describing $Z(A)$. Let $[D] \in \Cl(A)$, and set $D' = \Div(\alpha)(D)$. Since the map $\partial(A_{\sep}) \colon K_{s}^{\times} \to \Div(A_{\sep})$ is surjective, there exists $x \in K_{s}^{\times}$ such that $D' = \partial(A_{\sep})(x)$. One can accordingly define a cocycle $\sigma_{x} \colon \Gamma \to A_{\sep}^{\times}$ by setting $\sigma_{x}(\gamma) = \gamma(x)/x$ for $\gamma \in \Gamma$, and the class of $\sigma_{x}$ in $H^{1}(F, A_{\sep}^{\times}) = H^{1}(F, G_{\sep}^{\ast})$ does not depend on the choice of $x$. The map $Z(A)$ then takes $[D]$ to $[\sigma_{x}]$ in $H^{1}(F, G_{\sep}^{\ast})$.

\begin{thm}\label{GalFixedTorsors}
	Let $G$ be a geometrically integral, smooth group scheme over $F$. Put $A = F[G], K = \Frac(A)$, and $K_{s} = \Frac(A_{\sep})$. Suppose that $\Cl(A_{\sep}) = 0$, and $G_{\sep}^{\ast}[n] = 0$. Then the natural map $\Psi(-, n)(\alpha) \colon \Psi_{\nm}(A, n) \to \Psi_{\nm}(A_{\sep}, n)$ is an embedding of $\Psi_{\nm}(A, n)$ onto $H^{0}(F, \Psi_{\nm}(A_{\sep}, n))$.   
\end{thm}

\begin{proof}
	Put $\eta(A) = \dv_{n}(A) \circ \Omega_{n}(A)^{-1} \circ \Delta_{n}(A)$. By Proposition \ref{MuTorsExactSeq}, we have an exact sequence
	\[1 \to G^{\ast}/(G^{\ast})^{n} \xrightarrow{~\eta(A)~} \Div_{\nm}(A, n) \xrightarrow{~~} \Cl(A) \xrightarrow{~\cdot n ~} \Cl(A) \to 0.\]
	Because $G_{\sep}^{\ast}[n] = 0$, there is an exact sequence of $\Gamma$-modules
	\[1 \to G_{\sep}^{\ast} \xrightarrow{~ \cdot n ~} G_{\sep}^{\ast} \to G_{\sep}^{\ast}/(G_{\sep}^{\ast})^{n} \to 1 \]
	which yields the following long exact sequence in Galois cohomology:
	\[1 \to H^{0}(F, G_{\sep}^{\ast}) \xrightarrow{~ \cdot n ~} H^{0}(F, G_{\sep}^{\ast}) \to H^{0}(F, G_{\sep}^{\ast}/(G_{\sep}^{\ast})^{n}) \xrightarrow{~ \delta ~} H^{1}(F, G_{\sep}^{\ast}) \xrightarrow{~ \cdot n ~} H^{1}(F, G_{\sep}^{\ast}) \to \cdots \]
	We can rewrite the above (truncated) long exact sequence as 
	\[1 \to G^{\ast}/(G^{\ast})^{n} \to H^{0}(F, G_{\sep}^{\ast}/(G_{\sep}^{\ast})^{n}) \xrightarrow{~ \delta ~} H^{1}(F, G_{\sep}^{\ast}) \xrightarrow{~ \cdot n ~} H^{1}(F, G_{\sep}^{\ast}) .\]
	The boundary map $\delta$ can be described as follows: let $[u] \in H^{0}(F, G_{\sep}^{\ast}/(G_{\sep}^{\ast})^{n})$. Then $\gamma(u)/u \in (G_{\sep}^{\ast})^{n}$ for any $\gamma \in \Gamma$, so let $x_{\gamma}$ be the {\bf unique} element of $G_{\sep}^{\ast}$ such that $x_{\gamma}^{n} = \gamma(u)/u$. Then $\delta([u])$ is the class of the cocycle $\sigma_{u} \colon \Gamma \to G_{\sep}^{\ast}$ which sends $\gamma$ to $x_{\gamma}$. \\  
	Note that $\Cl(A_{\sep}) \cong \Pic(A_{\sep}) = 0$, and so $\Delta_{n}(A_{\sep})$ is an isomorphism by Corollary \ref{KMuNTorsors}. Let $\tau(A)$ denote the composition
	\[\Div_{\nm}(A, n) \xrightarrow{~ \Omega_{n}(A) \circ \dv_{n}(A)^{-1} ~} \Psi_{\nm}(A, n) \xrightarrow{~ \Psi(-, n)(\alpha)~} \Psi_{\nm}(A_{\sep}, n) \xrightarrow{~ \Delta_{n}(A_{\sep})^{-1}~} G_{\sep}^{\ast}/(G_{\sep}^{\ast})^{n}. \]
	Explicitly, given the class of a pair $(D, g)$ in $\Div_{\nm}(A, n)$, $D' := \Div(\alpha)(D)$ is principal, since $\Cl(A_{\sep}) = 0$, so there exists $x \in K_{s}^{\times}$ such that $\partial(A_{\sep})(x) = D'$; $\tau(A)$ sends $[(D, g)]$ to $[x^{n}/g] \in G_{\sep}^{\ast}/(G_{\sep}^{\ast})^{n}$. If the diagram 
	\begin{center}
		\begin{tikzcd}[column sep = .4 in, row sep = .5 in]
		1 \arrow{r} &  G^{\ast}/(G^{\ast})^{n} \arrow{d}{\Id_{G^{\ast}/(G^{\ast})^{n}}} \arrow{r}{\eta(A)} & \Div_{\nm}(A, n) \arrow{r}{} \arrow{d}{\tau(A)} & \Cl(A) \arrow{d}{Z(A)} \arrow{r}{\cdot n} & \Cl(A) \arrow{d}{Z(A)} \arrow{r} & 0  \\
		1 \arrow{r} & G^{\ast}/(G^{\ast})^{n} \arrow{r} & H^{0}(F,G_{\sep}^{\ast}/(G_{\sep}^{\ast})^{n}) \arrow{r}{\delta} & H^{1}(F,G_{\sep}^{\ast}) \arrow{r}{\cdot n} & H^{1}(F, G_{\sep}^{\ast}) \arrow{r} & 0 
		\end{tikzcd}
	\end{center}
	commutes, then $\tau(A)$ must be an isomorphism, so $\Psi(-, n)(\alpha)$ must be one as well. The last square is manifestly commutative. We have $\tau(A) \circ \eta(A) = \Delta_{n}(A_{\sep})^{-1} \circ \Psi(-, n)(\alpha) \circ \Delta_{n}(A)$, which is easily seen to be the inclusion $G^{\ast}/(G^{\ast})^{n} \to G_{\sep}^{\ast}/(G_{\sep}^{\ast})^{n}$. It remains to show that the middle square commutes. \\
	Let the pair $(D, g)$ represent a class in $\Div_{\nm}(A, n)$, and put $D' = \Div(\alpha)(D)$. Let $x \in K_{s}^{\times}$ such that $D' = \partial(A_{\sep})(x)$, so that $x^{n} = gu$ for some $u \in A_{\sep}^{\times}$. As explained above, $\tau(A)([(D, g)]) = [u] \in G_{\sep}^{\ast}/(G_{\sep}^{\ast})^{n}$. For any $\gamma \in \Gamma$, 
	\[\frac{\gamma(u)}{u} = \frac{\gamma(x^{n}g^{-1})}{x^{n}g^{-1}} = \frac{\gamma(x^{n})}{x^{n}} = \left(\frac{\gamma(x)}{x}\right)^{n}\]
	because $g^{-1}$ is $\Gamma$-invariant. Therefore, $\delta$ takes $[u]$ to the class of the cocycle $\sigma_{u} \colon \Gamma \to G_{\sep}^{\ast}$ defined by $\gamma \mapsto \gamma(x)/x$. On the other hand, as explained in paragraph immediately following Theorem \ref{ClassGroupCohomology}, $Z(A)$ takes $[D]$ to the very same cocycle class, so we're done. 
\end{proof}

\section{Type-Zero Invariants for Connected Reductive Groups} \label{TypeZeroInvariants}

Throughout this section, let $G$ be a smooth, connected, reductive algebraic group over $F$. Let $A = F[G], K = F(G)$, and let $\xi \colon A \to K$ denote the generic point of $G$. 

We are now equipped to determine the groups $\Inv_{\hom}^{0}(G, H)$ for $H = K^{M}_{1} \otimes_{\ZZ} \QQ/\ZZ$ and $H = K^{M}_{1}/n$ for all $n \in \NN$; this is the content of Theorems \ref{TypeZeroInvariantsGaloisQZ} and \ref{TypeZeroInvariantsReductiveModN} respectively. A key step is the observation that, under suitable conditions, a type-zero $H$-invariant of $G$ is determined by its value at $\xi$; precisely, the evaluation homomorphism $\ev_{\xi}(H) \colon \Inv^{0}_{\hom}(G, H) \to H(K)$ sending an invariant $I$ to $I(K)(\xi)$ is injective. Before proving this in Proposition \ref{TypeZeroGenericSet}, we need a technical lemma. For any positive integer $n$, let $p_{n} \colon G \to G$ denote the $n^{\thh}$ power map, which sends $x$ to $x^{n}$ for any $F$-algebra $R$ and any $x \in G(R)$. 

\begin{lem} \label{NthPowerDominant}
	The map $p_{n}$ is dominant.
\end{lem}

\begin{proof}
	Since the property of dominance descends under faithfully flat base change, we may assume that our base field $F$ is algebraically closed. By (e.g.) \cite[Theorem 17.44]{MilneAG}, the union of the Cartan subgroups of $G$ contains a dense open subset of $G$. Since $G$ is reductive, the Cartan subgroups of $G$ are precisely the maximal tori in $G$. But the restriction of $p_{n}$ to any torus in $G$ is surjective, and so the image of $p_{n}$ contains every torus in $G$.  
\end{proof}

\begin{prop}\label{TypeZeroGenericSet}
	Suppose $H$ is the $d^{\thh}$ graded component of a torsion cycle module. Let $I \in \Inv^{0}(G, H)$, and suppose $I(K)(\xi) = 1_{H(K)}$. Suppose that for any field extension $L/F$ and any $\alpha, \beta \in G(L)$, $I$ satisfies
	\[I(L)(\alpha)I(L)(\beta) = I(L)(\alpha\beta)I(L)(\varepsilon_{L}).\]
	Then $I$ is trivial. 
\end{prop}

\begin{proof}
	Let $L/F$ be a field extension, and fix $t \in G(L)$. Put $S := G_{L}$, and let $g \colon S \to G$ be the canonical base change morphism, with comorphism $f \colon A \to A_{L}$.  Let $E = L(S)$, and let $\xi' \colon A_{L} \to E$ be the generic point of $S$. Since $f$ is injective, the composition $\xi' \circ f$ extends to a morphism $u \colon K \to E$ of $F$-algebras such that $u \circ \xi = \xi' \circ f$. Put $\xi_{E} := u \circ \xi$, and let $n$ be a positive integer such that $I(E)(\xi_{E})^{n} = I(E)(\varepsilon_{E})^{n} = 1$. \\
	Suppose that there exist morphisms $i \colon K \to E, j \colon L \to E$ satisfying the following two properties:
	\begin{enumerate}[(a)]
		\item $H(j) \colon H(L) \to H(E)$ is injective;
		\item $G(i)(\xi) = (\xi_{E})^{n} \cdot t_{E}$, where $t_{E} := j \circ t$. 
	\end{enumerate}
	Then we have 
	\[H(j)(I(L)(t)) = I(E)(t_{E}) = I(E)(\xi_{E})^{n}I(E)(t_{E})I(E)(\varepsilon_{E})^{-n} = I(E)(({\xi}_{E})^{n} \cdot t_{E}), \]
	whence we conclude
	\[H(j)(I(L)(t)) = I(E)(G(i)(\xi)) = H(i)(I(K)(\xi)) = 1_{H(E)}.\]
	We therefore devote the remainder of the proof to constructing such a pair $(i, j)$. Let $j \colon L \to E$ denote the composition of the structural map $L \to A_{L}$ with $\xi'$. Since $S$ is a smooth algebraic $L$-variety such that $S(L) \neq \emptyset$, $H(j)$ is injective by \cite[Lemma 1.3]{MerkurjevIAG}. \\
	To construct $i$, let $s \colon A_{L} \to L$ be the unique $L$-algebra morphism such that $t = s \circ f = g(L)(t)$, and put $s_{E} = j \circ s$, so that $t_{E} = s_{E} \circ f$. Let $p_{n, s} \colon S \to S$ be the morphism of $L$-schemes given by the composition of the $n^{\thh}$ power map $p_{n}$ with right translation by $s$. By Corollary \ref{NthPowerDominant}, $p_{n, s}$ is dominant, and so the associated comorphism $h \colon A_{L} \to A_{L}$ is injective. In particular, the composition $\xi' \circ h$ extends to a morphism $v \colon E \to E$ of $L$-algebras such that $v \circ \xi' = \xi' \circ h$. Putting $i = v \circ u$,  we claim that $i$ satisfies (b). \\
	On the one hand, we have $p_{n, s}(E)(\xi') = \xi' \circ h$, but by definition of $p_{n, s}$, we also have $p_{n, s}(E)(\xi') = (\xi')^{n} \cdot s_{E}$. Accordingly, this yields 
	\[G(i)(\xi) = v \circ u \circ \xi = \xi' \circ h \circ f = g(E)(p_{n, s}(E)(\xi')).\]
	But we compute 
	\[g(E)(p_{n, s}(\xi')) = g(E)((\xi')^{n} \cdot s_{E}) = g(E)(\xi')^{n} \cdot g(E)(s_{E}) = (\xi_{E})^{n} \cdot t_{E},\]
	which establishes (b). 
	
\end{proof}

\begin{cor}\label{TypeZeroGenericGroup}
	The morphism $\ev_{\xi}(H) \colon \Inv^{0}_{\hom}(G, H) \to H(K)$ is injective. \qed
\end{cor}

For any fixed field extension $L/F$, there is a map $\Psi(A, n) \times G(L) \to L^{\times}/(L^{\times})^{n}$ which sends the pair $(X, y)$ to $\Delta_{n}(L)^{-1}(\Psi(-, n)(y)(X))$. If we fix a class $X \in \Psi(A, n)$ in the first argument, we obtain a set map $I_{X}(L) \colon G(L) \to L^{\times}/(L^{\times})^{n}$. As $L$ varies, the collection of maps $I_{X}$ determines an invariant in $\Inv^{0}(G, K^{M}_{1}/n)$. If $X$ is \emph{normalized}, then Corollary \ref{NormalizedGroupOp} shows that $I_{X}$ is homomorphic. We thus obtain a group homomorphism $\Lambda_{n}(G) \colon \Psi_{\nm}(A, n) \to \Inv^{0}_{\hom}(G, K^{M}_{1}/n)$. As the next theorem shows, $\Lambda_{n}(G)$ is in fact an isomorphism. 

\begin{thm} \label{TypeZeroInvariantsReductiveModN}
	The map 
	\[\Lambda_{n}(G) \colon \Psi_{\nm}(A, n) \to \Inv^{0}_{\hom}(G, K^{M}_{1}/n)\] 
	sending a class $X \in \Psi_{\nm}(A, n)$ to the invariant $I_{X}$ is an isomorphism. 
\end{thm}

\begin{proof}
	By Lemma \ref{TorsKerDivMap}, the map $\Delta_{n}(K)^{-1} \circ \Psi(-, n)(\xi) \colon \Psi(A, n) \to \ker(\partial_{n}(A))$ is an isomorphism. Thus, since $\ev_{\xi}(K^{M}_{1}/n) \circ \Lambda_{n}(G)$ coincides with the restriction of $\Delta_{n}(K)^{-1} \circ \Psi(-, n)(\xi)$ to $\Psi_{\nm}(A, n)$, $\Lambda_{n}(G)$ must be injective. \\
	Now, fix an invariant $I \in \Inv^{0}_{\hom}(G, K^{M}_{1}/n)$. By Corollary \ref{TypeZeroGenericGroup}, $\ev_{\xi}(K^{M}_{1}/n)$ is injective. The sequence
	\[\Inv^{0}_{\hom}(G, K^{M}_{1}/n) \xrightarrow{~\ev_{\xi}(K^{M}_{1}/n)~} K^{\times}/(K^{\times})^{n} \xrightarrow{~\partial_{n}(A)~} \Div(A)/n\Div(A) \]
	is a complex by \cite[Lemma 2.1]{MerkurjevIAG}, so $\ev_{\xi}(K^{M}_{1}/n)$ has image contained in $\ker(\partial_{n}(A))$. Letting $X \in \Psi(A, n)$ be a class such that $\Delta_{n}(K)^{-1}(\Psi(-, n)(\xi)(X)) = I(K)(\xi)$, we have $I_{X}(K)(\xi) = I(K)(\xi)$ by construction. We must therefore have $I_{X} = I$ by Theorem \ref{PointsMultTorsors} and Proposition \ref{TypeZeroGenericSet}. But as $I$ is homomorphic, it must be the case that $I_{X}(F)(\varepsilon_{F}) = I(F)(\varepsilon_{F})$ is the trivial class in $F^{\times}/(F^{\times})^{n}$, whence $X$ is normalized, and $\Lambda_{n}(G)(X) = I$.  
\end{proof}

\begin{cor} \label{TypeZeroInvariantsGaloisModN}
	Suppose that $G$ is a torus, and let $\alpha \colon A \to A_{\sep}$ be the canonical base change morphism. Then the map 
	\[(\Delta(A_{\sep}) \circ \Psi(-, n)(\alpha))^{-1} \circ \Lambda_{n}(G) \colon H^{0}(F, G_{\sep}^{\ast}/(G_{\sep}^{\ast})^{n}) \to \Inv^{0}_{\hom}(G, K^{M}_{1}/n)\]
	is an isomorphism. \qed 
\end{cor}

For any natural number $n$, let $\Inv^{0}(G, \iota_{n})$ denote the group morphism $\Inv^{0}_{\hom}(G, K^{M}_{1}/n) \to \Inv^{0}_{\hom}(G, K^{M}_{1} \otimes_{\ZZ} \QQ/\ZZ)$ given by composition with $\iota_{n}$. Likewise, if $n$ and $m$ are positive integers such that $n$ divides $m$, let $\Inv^{0}(G, \beta_{n, m})$ denote the group morphism $\Inv^{0}_{\hom}(G, K^{M}_{1}/n) \to \Inv^{0}_{\hom}(G, K^{M}_{1}/m)$ given by composition with $\beta_{n, m}$. Since $\iota_{n} = \iota_{m} \circ \beta_{n, m}$, we obtain a universal induced map 
\[\operatornamewithlimits{colim}\limits_{n \in \NN} \Inv^{0}(G, \iota_{n}) \colon \operatornamewithlimits{colim}\limits_{n \in \NN} \Inv^{0}_{\hom}(G, K^{M}_{1}/n) \to \Inv^{0}_{\hom}(G, K^{M}_{1} \otimes_{\ZZ} \QQ/\ZZ).\] 

\begin{prop} \label{ColimTypeZero}
	The map $\operatornamewithlimits{colim}\limits_{n \in \NN} \Inv^{0}(G, \iota_{n})$ is an isomorphism.
\end{prop}

\begin{proof}
	We have the following commutative diagram:
	\begin{center}
		\begin{tikzcd}
		{\displaystyle\operatornamewithlimits{colim}\limits_{n \in \NN} \Inv^{0}_{\hom}(G, K^{M}_{1}/n) } \arrow[rrr, "\operatornamewithlimits{colim}\limits_{n \in \NN} \ev_{\xi}(K^{M}_{1}/n)"] \arrow[ddd, "{\operatornamewithlimits{colim}\limits_{n \in \NN} \Inv^{0}(G, \iota_{n})}"'] &  &  & \displaystyle\operatornamewithlimits{colim}\limits_{n \in \NN} \ker(\partial_{n}(A)) \arrow[ddd, "\operatornamewithlimits{colim}\limits_{n \in \NN} \iota_{n}(K)"] \\
		&  &  &                                                                                         \\
		&  &  &                                                                                         \\
		{\Inv^{0}_{\hom}(G, K^{M}_{1} \otimes_{\ZZ} \QQ/\ZZ)} \arrow[rrr, "\ev_{\xi}(K^{M}_{1} \otimes_{\ZZ} \QQ/\ZZ)"']                                                         &  &  & \ker(\partial(A) \otimes_{\ZZ} \Id_{\QQ/\ZZ})                                          
		\end{tikzcd}
	\end{center}
	The rightmost arrow is an isomorphism, and the lower and upper horizontal arrows are injective by Corollary \ref{TypeZeroGenericGroup}, so it follows that $\operatornamewithlimits{colim}\limits_{n \in \NN} \Inv^{0}(G, \iota_{n})$ is injective. To see that $\operatornamewithlimits{colim}\limits_{n \in \NN} \Inv^{0}(G, \iota_{n})$ is surjective, fix an invariant $I \in \Inv^{0}_{\hom}(G, K^{M}_{1} \otimes_{\ZZ} \QQ/\ZZ)$ and let $x = I(K)(\xi) \in \ker(\partial(A) \otimes_{\ZZ} \Id_{\QQ/\ZZ})$. There exists some positive integer $n$ and $y \in \ker(\partial_{n}(A))$ such that $\iota_{n}(K)(y) = x$. Let $Y \in \Psi(A, n)$ with $\Psi(-, n)(\xi)(Y) = \Delta_{n}(K)(y)$. Then the associated invariant $I_{Y} \in \Inv^{0}(G, K^{M}_{1}/n)$ satisfies $(\iota_{n} \circ I_{Y})(K)(\xi) = x = I(K)(\xi)$, and so $\iota_{n} \circ I_{Y} = I$ by Theorem \ref{PointsMultTorsors} and Proposition \ref{TypeZeroGenericSet}. In particular, $((\iota_{n} \circ I_{Y})(F))(\varepsilon_{F})$ is the trivial class in $F^{\times} \otimes_{\ZZ} \QQ/\ZZ$, which means that $z := I_{Y}(F)(\varepsilon_{F})$ belongs to the kernel of $\iota_{n}(F) \colon F^{\times}/(F^{\times})^{n} \to F^{\times} \otimes_{\ZZ} \QQ/\ZZ$. \\
	This can only be the case if $z \in \ker(\beta_{n, nd}(F))$ for some $d \in \NN$, so fix such a $d$. For any field extension $M/F$, the diagram
	\begin{center}
		\begin{tikzcd}
		M^{\times}/(M^{\times})^{n} \arrow[ddd, "{\beta_{n, nd}(M)}"'] \arrow[rrr, "\Delta_{n}(M)"] &  &  & {\Psi(M, n)} \arrow[ddd, "{\omega_{n, nd}(M)}"] \\
		&  &  &                                                 \\
		&  &  &                                                 \\
		M^{\times}/(M^{\times})^{nd} \arrow[rrr, "\Delta_{nd}(M)"']                                &  &  & {\Psi(M, nd)}                                  
		\end{tikzcd}
	\end{center}
	commutes, and so putting $Y' = \omega_{n, nd}(A)(Y)$, $\Psi(-, nd)(\varepsilon_{F})(Y') = \Delta_{nd}(F)(\beta_{n, nd}(F)(z))$, whence $Y'$ is normalized. Thus, $I_{Y'} = \Lambda_{nd}(G)(Y')$ is homomorphic, and 
	\[((\iota_{nd} \circ I_{Y'})(K))(\xi)) = \iota_{nd}(K)(\beta_{n, nd}(K)(y)) = \iota_{n}(K)(y) = x,\]
	so $\iota_{nd} \circ I_{Y'} = I$ by Corollary \ref{TypeZeroGenericGroup}. 
\end{proof}

\begin{cor} \label{TypeZeroInvariantsGaloisQZ}
	If $G$ is a torus, then $\Inv^{0}_{\hom}(G, K^{M}_{1} \otimes_{\ZZ} \QQ/\ZZ) \cong H^{0}(F, G_{\sep}^{\ast} \otimes \QQ/\ZZ)$. 
\end{cor}

\begin{proof}
	If $\alpha \colon A \to A_{\sep}$ denotes the canonical base change morphism, this follows from Proposition \ref{ColimTypeZero}, Theorem \ref{GalFixedTorsors}, and the fact that the diagram
	\begin{center}
		\begin{tikzcd}
		{H^{0}(F, G^{\ast}_{\sep}/(G^{\ast}_{\sep})^{n})} \arrow[ddd, "\cdot {m/n}"] &  &  & {\Psi_{\nm}(A, n)} \arrow[lll, "{\Delta_{n}(A_{\sep})^{-1} \circ \Psi(-, n)(\alpha)}"'] \arrow[rrr, "\Lambda_{n}(G)"] \arrow[ddd, "{\omega_{n, m}(A)}"] &  &  & {\Inv^{0}_{\hom}(G, K^{M}_{1}/n)} \arrow[ddd, "{\Inv^{0}(G, \beta_{n, m})}"] \\
		&  &  &                                                                                                                                                         &  &  &                                                                       \\
		&  &  &                                                                                                                                                         &  &  &                                                                       \\
		{H^{0}(F, G^{\ast}_{\sep}/(G_{\sep}^{\ast})^{m})}                            &  &  & {\Psi_{\nm}(A, m)} \arrow[lll, "{\Delta_{m}(A_{\sep})^{-1} \circ \Psi(-, m)(\alpha)}"] \arrow[rrr, "\Lambda_{m}(G)"']                                   &  &  & {\Inv^{0}_{\hom}(G, K^{M}_{1}/m)}                                    
		\end{tikzcd}
	\end{center}
	commutes for all $n, m \in \NN$ with $n$ dividing $m$.
\end{proof}

\section{Computation of Degree One Milnor $K$-invariants of Groups of Multiplicative Type} \label{TypeOneInvariants}

In this section, we determine the degree one Milnor $K$-invariants of an algebraic group $G$ of multiplicative type. To begin, fix a resolution \ref{ResolutionByTori} of $G$ by tori. Applying the snake lemma to the diagram
\begin{center}
	\begin{tikzcd}[column sep = .5 in, row sep = .5 in]
	1 \arrow{r} & T_{\sep}^{\ast} \arrow{d}{\cdot n}  \arrow{r}{g_{\sep}^{\ast}} & P_{\sep}^{\ast} \arrow{d}{\cdot n} \arrow{r}{f_{\sep}^{\ast}} & G_{\sep}^{\ast} \arrow{d}{\cdot n} \arrow{r} & 1 \\
	1 \arrow{r} & T_{\sep}^{\ast} \arrow{r}{g_{\sep}^{\ast}} & P_{\sep}^{\ast} \arrow{r}{f_{\sep}^{\ast}} & G_{\sep}^{\ast} \arrow{r} & 1 	
	\end{tikzcd}
\end{center}
yields the exact sequence of $\Gamma$-modules
\[1 \to G_{\sep}^{\ast}[n] \xrightarrow{~~} T_{\sep}^{\ast}/(T_{\sep}^{\ast})^{n} \xrightarrow{~~} P_{\sep}^{\ast}/(P_{\sep}^{\ast})^{n}, \]
and after taking $\Gamma$-fixed points we obtain the exact sequence
\[1 \to H^{0}(F, G_{\sep}^{\ast}[n]) \xrightarrow{~~} H^{0}(F, T_{\sep}^{\ast}/(T_{\sep}^{\ast})^{n}) \xrightarrow{~~} H^{0}(F, P_{\sep}^{\ast}/(P_{\sep}^{\ast})^{n}) \]
of abelian groups. Let $A = F[G], B = F[P], C = F[T]$, let $g^{\sharp} \colon C \to B, f^{\sharp} \colon B \to A$ be the associated comorphisms, and let $\alpha_{X} \colon X \to X_{\sep}$ denote the canonical base change morphism for $X = A, B, C$. For $Y = B, C$, let $\ell_{n}(Y) := \Delta_{n}(Y_{\sep})^{-1} \circ \Psi(-, n)(\alpha_{Y}) \circ \lambda_{n}(Y)^{-1}$. 

\begin{prop} \label{ExactSeqCharsTors}
	The diagram
	\begin{center}
		\begin{tikzcd}
		{G^{\ast}[n]} \arrow[rrr, "\upsilon_{n}(G)"] \arrow[ddd] &  &  & {\Tors_{\nm}(T, \bmu_{n, F}}) \arrow[rrr, "{\Tors^{\ast}(g)(\bmu_{n, F})}"] \arrow[ddd, "{\ell_{n}(C)}"'] &  &  & {\Tors_{\nm}(P, \bmu_{n, F})} \arrow[ddd, "{\ell_{n}(B)}"'] \\
		&  &  &                                                                                                                                          &  &  &                                                                                       \\
		&  &  &                                                                                                                                          &  &  &                                                                                       \\
		{H^{0}(F, G_{\sep}^{\ast}[n])} \arrow[rrr]               &  &  & {H^{0}(F, T_{\sep}^{\ast}/(T_{\sep}^{\ast})^{n})} \arrow[rrr, "\mathcal{K}^{n}(g_{\sep}^{\sharp})"']                                                                            &  &  & {H^{0}(F, P_{\sep}^{\ast}/(P_{\sep}^{\ast})^{n})}                                    
		\end{tikzcd}
	\end{center}
	commutes. 
\end{prop}  

\begin{proof}
	The right square commutes because $\Delta_{n}$ and $\lambda_{n}$ are natural transformations. To see that the left square commutes, fix $\chi \in G^{\ast}[n]$. 
	Consider the commutative diagram
	\begin{center}
		\begin{tikzcd}
		1 \arrow[rr] &  & T_{\sep}^{\ast} \arrow[rr, "j_{\sep}^{\ast}"] \arrow[dd, "\Id_{T_{\sep}^{\ast}}"'] &  & P_{\sep}^{\ast} \times_{G_{\sep}^{\ast}} \mathbb{Z}/n\mathbb{Z} \arrow[rr, "(\pi_{n})_{\sep}^{\ast}"] \arrow[dd, "(\pi_{P})_{\sep}^{\ast}"] &  & \mathbb{Z}/n\mathbb{Z} \arrow[rr] \arrow[dd, "\chi_{\sep}^{\ast}"] &  & 1 \\
		&  &                                                                      &  &                                                                                                             &  &                                                             &  &   \\
		1 \arrow[rr] &  & T_{\sep}^{\ast} \arrow[rr, "g_{\sep}^{\ast}"']                          &  & P_{\sep}^{\ast} \arrow[rr, "f_{\sep}^{\ast}"']                                                                &  & G_{\sep}^{\ast} \arrow[rr]                                  &  & 1
		\end{tikzcd}
	\end{center}
	of $\Gamma$-modules with exact rows, and let $H$ denote the group of multiplicative type dual to $P_{\sep}^{\ast} \times_{G_{\sep}^{\ast}} \mathbb{Z}/n\mathbb{Z}$. The $F$-group morphism $j \colon H \to T$ dual to $j_{\sep}^{\ast}$ is a $\mu_{n, F}$-torsor over $T$, and we claim that $j$ represents the class $\upsilon_{n}(G)(\chi)$. Indeed, let $\pi_{n} \colon \bmu_{n, F} \to H, \pi_{P} \colon P \to H$ be the morphisms dual to $(\pi_{n})_{\sep}^{\ast}$ and $(\pi_{P})_{\sep}^{\ast}$ respectively. The morphism of $T$-schemes $P \times \bmu_{n, F} \to H$ defined on $R$-points by $(x, y) \mapsto \pi_{P}(R)(x)\pi_{n}(R)(y)$ for any $F$-algebra $R$ and any $x \in P(R), y \in \bmu_{m, F}(R)$ is constant on $G^{\chi}$-orbits. It therefore descends to a universal map $(P \times \bmu_{n, F})/G^{\chi} \to H$ over $T$, which one may check is $\bmu_{n, F}$-equivariant. \\
	Now, let $y \in P_{\sep}^{\ast}$ be such that $f_{\sep}^{\ast}(y) = \chi$, and let $z \in T_{\sep}^{\ast}$ be such that $g_{\sep}^{\ast}(z) = y^{n}$. We must show that $\Spec(C_{\sep}[X]/\langle X^{n} - z \rangle) \to \Spec(C_{\sep})$ and $j_{\sep}$ are isomorphic as $\bmu_{n, F_{\sep}}$-torsors over $T_{\sep}$. Equivalently, we must exhibit a(n) (iso)morphism of $\mathbb{Z}/n\mathbb{Z}$-graded $C_{\sep}$-algebras $s \colon C_{\sep}[X]/\langle X^{n}-z \rangle \to F_{\sep}[H_{\sep}]$. The condition that $s$ respect the $\mathbb{Z}/n\mathbb{Z}$-grading ensures that the dual morphism of schemes $H \to \Spec(C_{\sep}[X]/\langle X^{n} - z \rangle)$ is $\bmu_{n, F_{\sep}}$-equivariant, hence an isomorphism of $\bmu_{n, F_{\sep}}$-torsors. \\
	By construction, $F_{\sep}[H_{\sep}]$ is the group algebra of $H_{\sep}^{\ast} = P_{\sep}^{\ast} \times_{G_{\sep}^{\ast}} \ZZ/n\ZZ$ over $F_{\sep}$, and $C_{\sep}$ is likewise the group algebra $F_{\sep} \langle T_{\sep}^{\ast} \rangle$. The comorphism $j_{\sep}^{\sharp}$ corresponds to the $\Gamma$-module embedding $j_{\sep}^{\ast} \colon T_{\sep}^{\ast} \hookrightarrow H_{\sep}^{\ast}$. For each $v \in \ZZ/n\ZZ$, put $Q_{v} := ((\pi_{n})_{\sep}^{\ast})^{-1}(v)$. Note that $Q_{v}Q_{v'} \subset Q_{v+v'}$, and $Q_{v} = (y, [1]_{n})^{k_{v}}j^{\ast}(T_{\sep}^{\ast})$, where $k_{v} \in \NN$ is the unique representative for $v$ between $0$ and $n-1$. The $(\ZZ/n\ZZ)$-grading on $H_{\sep}$ arises from the partition 
	\[H_{\sep}^{\ast} = \coprod_{v \in \ZZ/n\ZZ} Q_{v}\]
	by setting $R_{v}$ to be the $F_{\sep}$-subspace of $F_{\sep}[H_{\sep}]$ generated by $Q_{v}$. We clearly have $F_{\sep}[H_{\sep}] = \bigoplus_{v \in \ZZ/n\ZZ} R_{v}$, and $R_{v}R_{v'} \subset R_{v+v'}$ follows from $Q_{v}Q_{v'} \subset Q_{v+v'}$. Furthermore, $R_{v}$ is the $C_{\sep}$-submodule of $F_{\sep}[H_{\sep}]$ generated by $(y, [1]_{n})^{k_{v}}$. With this in mind, let $s \colon C_{\sep}[X]/\langle X^{n}-z\rangle \to F_{\sep}[H_{\sep}]$ be the universal morphism of $C_{\sep}$-algebras sending the class of $X$ to $(y, [1]_{n})$. This respects the $(\ZZ/n\ZZ)$-grading on each $C_{\sep}$-algebra, since $(y, [1]_{n})$ belongs to the $[1]_{n}$-graded component of $H_{\sep}$, and $C_{\sep}$ embeds into each algebra as the $[0]_{n}$-graded component.  
\end{proof}

Since all vertical arrows of the diagram in Proposition \ref{ExactSeqCharsTors} are isomorphisms, this proves:

\begin{cor}
	The sequence
	\[1 \to G^{\ast}[n] \xrightarrow{~\upsilon_{n}(G)~} \Tors_{\nm}(T, \bmu_{n, F}) \xrightarrow{~\Tors^{\ast}(g)(\bmu_{n, F})~} \Tors_{\nm}(P, \bmu_{n, F})\]
	is exact. \qed
\end{cor}

For any smooth, connected, reductive group $R$ over $F$, define $\tilde{\Lambda}_{n}(R) \colon \Tors_{\nm}(R, \bmu_{n, F}) \to \Inv^{0}_{\hom}(R, K^{M}_{1}/n)$ by $\tilde{\Lambda}_{n}(R) = \Lambda_{n}(R) \circ \lambda_{n}(F[R])^{-1}$. As noted in section \ref{Pushforward}, the last crucial detail in our computation of $\Inv^{1}_{\hom}(G, K^{M}_{1}/n)$ is the following lemma. 

\begin{lem} \label{KeyDiagramCommutes}
The diagram
\begin{center}
	\begin{tikzcd}
	{G^{\ast}[n]} \arrow[rr, "\upsilon_{n}(G)"] \arrow[dd, "{\Phi(G, n)}"] &  & {\Tors_{\nm}(T, \bmu_{n, F})} \arrow[rr, "{\Tors^{\ast}(g)(\bmu_{n, F})}"] \arrow[dd, "\tilde{\Lambda}_{n}(T)"] &  & {\Tors_{\nm}(P, \bmu_{n, F})} \arrow[dd, "\tilde{\Lambda}_{n}(P)"] \\
	&  &                                                                        &  &                                                                                                                 &  &                                                                    \\
	{\Inv^{1}_{\hom}(G, K^{M}_{1}/n)} \arrow[rr, "{\Inv(\rho, K^{M}_{1}/n)}"']            &  & {\Inv^{0}_{\hom}(T, K^{M}_{1}/n)} \arrow[rr, "{\Inv(g, K^{M}_{1}/n)}"']                                                        &  & {\Inv^{0}_{\hom}(P, K^{M}_{1}/n)}                                 
	\end{tikzcd}
\end{center}
commutes. 
\end{lem}

\begin{proof} 
	Unwinding the definitions of $\tilde{\Lambda}_{n}(T)$ and $\tilde{\Lambda}_{n}(P)$, one sees that the commutativity of the right square is a consequence of the functoriality of the pullback map on torsors. To be precise, if $\alpha \colon Y \to X, \beta \colon Z \to Y$ are morphisms of $F$-schemes, then $\Tors^{\ast}(\alpha \circ \beta) = \Tors^{\ast}(\beta) \circ \Tors^{\ast}(\alpha)$. The left square commutes because pullback operation on torsors commutes with changing the group.  
\end{proof}

As noted at the end of section \ref{Pushforward}, after a diagram chase, this proves:

\begin{thm} \label{TypeOneGMTModN}
	The map $\Phi(G, n) \colon G^{\ast}[n] \to \Inv^{1}_{\hom}(G, K^{M}_{1}/n)$ is an isomorphism. \qed  
\end{thm}

As was the case for type-zero invariants, for any natural number $n$, there is a group morphism $\Inv^{1}(G, \iota_{n}) \colon \Inv^{1}_{\hom}(G, K^{M}_{1}/n) \to \Inv^{1}_{\hom}(G, K^{M}_{1} \otimes_{\ZZ} \QQ/\ZZ)$ given by composition with $\iota_{n}$. For positive integers $n, m$ with $n$ dividing $m$, the maps $\Inv^{1}(G, \iota_{n}), \Inv^{1}(G, \iota_{m})$ are compatible with the map $\Inv^{1}(G, \beta_{n, m}) \colon \Inv^{1}_{\hom}(G, K^{M}_{1}/n) \to \Inv^{1}_{\hom}(G, K^{M}_{1}/m)$ given by composition with $\beta_{n, m}$, and so we obtain a universal induced map 
\[\operatornamewithlimits{colim}\limits_{n \in \NN} \Inv^{1}(G, \iota_{n}) \colon \operatornamewithlimits{colim}\limits_{n \in \NN} \Inv^{1}_{\hom}(G, K^{M}_{1}/n) \to \Inv^{1}_{\hom}(G, K^{M}_{1} \otimes_{\ZZ} \QQ/\ZZ).\] 

\begin{prop} \label{ColimTypeOne}
	The map $\operatornamewithlimits{colim}\limits_{n \in \NN} \Inv^{1}(G, \iota_{n})$ is an isomorphism.
\end{prop}

\begin{proof}
	Set
	\[u = \displaystyle \operatornamewithlimits{colim}\limits_{n \in \NN} \Inv(\rho, K^{M}_{1}/n), v = \displaystyle \operatornamewithlimits{colim}\limits_{n \in \NN} \Inv(g, K^{M}_{1}/n), \]
	\[u' = \Inv(\rho, K^{M}_{1} \otimes_{\ZZ} \QQ/\ZZ), v' = \Inv(g, K^{M}_{1} \otimes_{\ZZ} \QQ/\ZZ).\]
	We have a commutative diagram
	\begin{center}
		\begin{tikzcd} [column sep = 0.2 in]
		{\displaystyle \operatornamewithlimits{colim}\limits_{n \in \NN} \Inv^{1}_{\hom}(G, K^{M}_{1}/n)} \arrow[rrr, "u"] \arrow[ddd, "{\operatornamewithlimits{colim}\limits_{n \in \NN} \Inv^{1}(G, \iota_{n})}"'] &  &  & {\displaystyle \operatornamewithlimits{colim}\limits_{n \in \NN} \Inv^{0}_{\hom}(T, K^{M}_{1}/n)} \arrow[rrr, "v"] \arrow[ddd, "{\operatornamewithlimits{colim}\limits_{n \in \NN} \Inv^{0}(T, \iota_{n})}"'] &  &  & {\displaystyle \operatornamewithlimits{colim}\limits_{n \in \NN} \Inv^{0}_{\hom}(P, K^{M}_{1}/n)} \arrow[ddd, "{\operatornamewithlimits{colim}\limits_{n \in \NN} \Inv^{0}(P, \iota_{n})}"'] \\
		&  &  &                                                                                                                                                                                     &  &  &                                                                                                                                                                                  &  &  &                                                                                                                           \\
		&  &  &                                                                                                                                                                                     &  &  &                                                                                                                                                                                  &  &  &                                                                                                                           \\
		{\Inv^{1}_{\hom}(G, K^{M}_{1} \otimes_{\ZZ} \QQ/\ZZ)} \arrow[rrr, "u'"']                                                                   &  &  & {\Inv^{0}_{\hom}(T, K^{M}_{1} \otimes_{\ZZ} \QQ/\ZZ)} \arrow[rrr, "v'"']                                                                   &  &  & {\Inv^{0}_{\hom}(P, K^{M}_{1} \otimes_{\ZZ} \QQ/\ZZ)}                                                                    
		\end{tikzcd}
	\end{center}
	whose rows are exact. Since $\operatornamewithlimits{colim}\limits_{n \in \NN} \Inv^{0}(T, \iota_{n})$ and $\operatornamewithlimits{colim}\limits_{n \in \NN} \Inv^{0}(P, \iota_{n})$ are isomorphisms by Proposition \ref{ColimTypeZero}, and $u, u'$ are injective, $\operatornamewithlimits{colim}\limits_{n \in \NN} \Inv^{1}(G, \iota_{n})$ is an isomorphism. 
\end{proof}

\begin{thm} \label{TypeOneGMTQZ}
	The map $\Phi(G) \colon G^{\ast}_{\tors} \to \Inv^{1}_{\hom}(G, K^{M}_{1} \otimes_{\ZZ} \QQ/\ZZ)$ is a group isomorphism.  
\end{thm}

\begin{proof}
	Let $n, m$ be positive integers with $n$ dividing $m$, and let $\tau_{n, m} \colon \bmu_{n, F} \to \bmu_{m, F}$ be the canonical embedding. We claim that the diagram 
	\begin{center}
		\begin{tikzcd}
		{G^{\ast}[n]} \arrow[ddd, "{\Phi(G, n)}"'] \arrow[rrr, "{\sigma_{n, m}}"]     &  &  & {G^{\ast}[m]} \arrow[ddd, "{\Phi(G, m)}"] \\
		&  &  &                                           \\
		&  &  &                                           \\
		{\Inv^{1}_{\hom}(G, K^{M}_{1}/n)} \arrow[rrr, "{\Inv^{1}(G, \beta_{n, m})}"'] &  &  & {\Inv^{1}_{\hom}(G, K^{M}_{1}/m)}        
		\end{tikzcd}
	\end{center}
	commutes, where $\sigma_{n, m}$ is the group morphism given by composition with $\tau_{n, m}$. Indeed, it is sufficient to show that $\Sigma_{m}(L) \circ \beta_{n, m}(L) = \Tors_{\ast}(\tau_{n, m})(L) \circ \Sigma_{n}(L)$ for any field extension $L/F$. Fixing $[y] \in L$, put $U = \Spec(L[X]/\langle X^{n} -y \rangle), V = \Spec(L[X]/\langle X^{m} - y^{m/n} \rangle)$. The morphism of $L$-schemes $U \times \bmu_{m, L} \to V$ defined functorially by
	\[U(R) \times \bmu_{m, L}(R) \to V(R), (u, z) \mapsto uz\]
	for any $L$-algebra $R$ is constant on $\bmu_{n, L}^{\tau_{n, m}}$-orbits, and so descends to a morphism of $L$-schemes $(U \times \bmu_{m, L})/(\bmu_{n, L}^{\tau_{n, m}}) \to V$, which one may check is $\bmu_{m, L}$-equivariant. This establishes that $\Tors_{\ast}(\tau_{n, m})(L)(U) = V$. \\
	The universally induced map $\operatornamewithlimits{colim}\limits_{n \in \NN} \Phi(G, n) \colon G^{\ast}_{\tors} \to \operatornamewithlimits{colim}\limits_{n \in \NN} \Inv^{1}(G, K^{M}_{1}/n)$ is an isomorphism, as $\Phi(G, n)$ is an isomorphism for each $n$. Since $\Phi(G)$ is just the composition of $\operatornamewithlimits{colim}\limits_{n \in \NN} \Phi(G, n)$ with the $\operatornamewithlimits{colim}\limits_{n \in \NN} \Inv^{1}(G, \iota_{n})$, it is an isomorphism by Proposition \ref{ColimTypeOne}.  
\end{proof}


\bibliographystyle{amsalpha}
\bibliography{References}

\end{document}